\documentclass[11pt]{amsart}%
\usepackage{palatino, mathpazo}
\usepackage{amsfonts}
\usepackage{amsmath}
\usepackage{amssymb,latexsym}
\usepackage{graphicx}
\usepackage[mathscr]{eucal}
\usepackage{amssymb}%
\providecommand{\U}[1]{\protect \rule{.1in}{.1in}}

\newtheorem{theorem}{Theorem}[section]

\newtheorem{conjecture}[theorem]{Conjecture}
\newtheorem{corollary}[theorem]{Corollary}

\newtheorem{lemma}[theorem]{Lemma}

\newtheorem{proposition}[theorem]{Proposition}
\theoremstyle{remark}
\newtheorem{remark}[theorem]{Remark}

\numberwithin{equation}{section}
\begin{document}
\title[Curvature equations on rectangular tori]{Sharp nonexistence results for curvature equations with four singular sources on
rectangular tori}
\author{Zhijie Chen}
\address{Department of Mathematical Sciences, Yau Mathematical Sciences Center,
Tsinghua University, Beijing, 100084, China }
\email{zjchen2016@tsinghua.edu.cn}
\author{Chang-Shou Lin}
\address{Taida Institute for Mathematical Sciences (TIMS), Center for Advanced Study in
Theoretical Sciences (CASTS), National Taiwan University, Taipei 10617, Taiwan }
\email{cslin@math.ntu.edu.tw}

\begin{abstract}
In this paper, we prove that there are no solutions for the curvature equation
\[
\Delta u+e^{u}=8\pi n\delta_{0}\text{ on
}E_{\tau}, \quad n\in\mathbb{N},
\]
where $E_{\tau}$ is a flat rectangular torus and $\delta_{0}$ is the Dirac measure at
the lattice points. This confirms a conjecture in
\cite{CLW2} and also improves a result of Eremenko and Gabrielov \cite{EG}. The nonexistence is a delicate problem because the equation always has solutions if $8\pi n$ in the RHS is replaced by $2\pi \rho$ with $0<\rho\notin 4\mathbb{N}$. Geometrically, our result implies that a rectangular torus $E_{\tau}$ admits a metric with curvature $+1$ acquiring a conic singularity at the lattice points with angle $2\pi\alpha$ if and only if $\alpha$ is not an odd integer.

Unexpectedly,
our proof of the nonexistence result is to apply the spectral theory of finite-gap potential, or equivalently the algebro-geometric solutions of stationary KdV hierarchy equations. Indeed, our proof can also yield a sharp nonexistence result for the curvature equation with singular sources at three half periods and the lattice points.
\end{abstract}
\maketitle

\section{Introduction}

Throughout the paper, we use the notations $\omega_{0}=0$, $\omega_{1}=1$,
$\omega_{2}=\tau$, $\omega_{3}=1+\tau$ and $\Lambda_{\tau}=\mathbb{Z+Z}\tau$,
where $\tau \in \mathbb{H}=\left \{  \tau|\operatorname{Im}\tau>0\right \}  $.
Define $E_{\tau}:=\mathbb{C}/\Lambda_{\tau}$ to be a flat torus in the plane
and $E_{\tau}[2]:= \{ \frac{\omega_{k}}{2}|0\leq k\leq3\}+\Lambda_{\tau}$ to
be the set consisting of the lattice points and $2$-torsion points in
$E_{\tau}$. Consider the following curvature equation with \textit{four
singular sources}:%
\begin{equation}
\Delta u+e^{u}=8\pi \sum_{k=0}^{3}n_{k}\delta_{\frac{\omega_{k}}{2}}\quad\text{ on
}\; E_{\tau}, \label{mean}%
\end{equation}
where $\delta_{\omega_{k}/2}$ is the Dirac measure at $\frac{\omega_{k}}{2}$,
and $n_{k}\in \mathbb{Z}_{\geq 0}$ for all $k$ with $\sum n_{k}\geq1$. By changing variable $z\mapsto z+\frac{\omega_k}{2}$, we can always assume $n_{0}=\max_{k}n_{k}\geq1$.

Not surprisingly, (\ref{mean}) is related to various research areas. In
conformal geometry, a solution $u$ to (\ref{mean}) leads to a metric
$ds^{2}=\frac{1}{2}e^{u}(dx^{2}+dy^{2})$ with constant Gaussian curvature $+1$
acquiring \textit{conic singularities at $\frac{\omega_{k}}{2}$'s}. It also
appears in statistical physics as the equation for the \textit{mean field
limit} of the Euler flow in Onsager's vortex model (cf. \cite{CLMP}), hence
also called a \textit{mean field equation}. Recently equation (\ref{mean}) was shown to be related to the
self-dual condensates of the Chern-Simons-Higgs equation in superconductivity.
We refer the readers to \cite{CLW4,Choe,EG,LY,NT1} and references therein for
recent developments of related subjects of equation (\ref{mean}).

The existence of solutions of equation (\ref{mean}) is very challenging from
the PDE point of view. In fact, \emph{the solvability of (\ref{mean})
essentially depends on the moduli $\tau$ in a sophisticated manner}. This
phenomena was first discovered by Wang and the second author \cite{LW} when
they studied the case $n_{0}=1$ and $n_{1}=n_{2}=n_{3}=0$, i.e.
\begin{equation}
\Delta u+e^{u}=8\pi \delta_{0}\  \  \text{on}\ E_{\tau}. \label{eq1-1}%
\end{equation}
For example, they proved that when $\tau \in i\mathbb{R}_{>0}$ (i.e. $E_{\tau}$
is a rectangular torus), equation (\ref{eq1-1}) has \textit{no} solution;
while for $\tau=\frac{1}{2}+\frac{\sqrt{3}}{2}i$ (i.e. $E_{\tau}$ is a rhombus
torus), equation (\ref{eq1-1}) has solutions. Later, equation (\ref{eq1-1})
was thoroughly investigated in \cite{CKLW,LW4}.

For the case $n_{0}=n\geq2$ and $n_{1}=n_{2}=n_{3}=0$, i.e.
\begin{equation}
\Delta u+e^{u}=8n\pi \delta_{0}\  \  \text{on}\ E_{\tau}, \label{eq1}%
\end{equation}
Chai, Lin, Wang \cite{CLW} and subsequently Lin, Wang \cite{CLW2} studied it
from the viewpoint of algebraic geometry. They developed a theory to connect
this PDE problem with hyperelliptic curves and modular forms. Among other
things, they proposed the following conjecture.\medskip

\noindent \textbf{Conjecture. }\cite{Lin-CDM,CLW2} \textit{When $\tau \in i\mathbb{R}%
_{>0}$, i.e. $E_{\tau}$ is a rectangular torus, equation (\ref{eq1}) has no
solutions for any $n\geq2$.} \medskip

Geometrically, this conjecture is equivalent to assert that the rectangular torus admits no metric with constant curvature $1$ and a conical singularity with angle $2\pi (1+2n)$.
Recently in \cite{CKL}, we proved this conjecture for $n=2$. However, this
proof can not work for general $n\ge3$. One main purpose of this paper is to
resolve this conjecture via a completely different idea.

\begin{theorem}
\label{no-sol-n} If $\tau \in i\mathbb{R}_{>0}$, i.e. $E_{\tau}$ is a
rectangular torus, then equation (\ref{eq1}) has no solutions for any $n\geq1$.
\end{theorem}

There are two important consequences of Theorem \ref{no-sol-n}. One is related to the pre-modular form $Z_{r,s}^{(n)}(\tau)$ associated with the Lam\'{e} equation introduced by \cite{CLW2}. See Section \ref{premodularform}. The other is related to the following conjecture.

\begin{conjecture}\label{con-n-solution}
Suppose $\tau\in i\mathbb{R}_{>0}$ and $\rho\in (8\pi (n-1), 8\pi n)$, $n\in \mathbb{N}$. Then the equation
\[\Delta u+ e^u=\rho \delta_0\quad \text{on}\; E_{\tau}\]
possesses exactly $n$ solutions. \end{conjecture}

Conjecture \ref{con-n-solution} was already proved for $\rho\in (0,8\pi)$ in \cite{LW4} and for $\rho=8\pi (n-\frac{1}{2})$ in \cite{CLW}.
In \cite{Chen-Lin2017}, we will apply Theorem \ref{no-sol-n} to prove Conjecture \ref{con-n-solution} for $|\rho-8\pi n|\ll 1$ and $|\rho-8\pi (n-1)|\ll 1$.

In fact, our proof of Theorem \ref{no-sol-n} also works for equation
(\ref{mean}) with more general $n_{k}$'s. Our first main result of this paper
is the following sharp nonexistence result.

\begin{theorem}
\label{no-even-solution} Let $n_{k}\in \mathbb{Z}_{\geq 0}$ for all $k$ with
$\max_{k} n_{k}\geq1$. If $(n_{0},n_{1},n_{2},n_{3})$ satisfies neither \begin{equation}
\label{c1}\frac{n_{1}+n_{2}-n_{0}-n_{3}}{2}\geq1,\quad n_{1}\geq1,\quad
n_{2}\geq1
\end{equation}
nor
\begin{equation}
\label{c2}\frac{n_{1}+n_{2}-n_{0}-n_{3}}{2}\leq-1,\quad n_{0}\geq1,\quad
n_{3}\geq1,
\end{equation} then for any $\tau \in i\mathbb{R}_{>0}$, equation
(\ref{mean}) on $E_{\tau}$ has no even solutions.
\end{theorem}

\begin{remark}
It was proved in \cite{CLW} that once (\ref{eq1}) has a solution, then it has
also an even solution. Thus Theorem \ref{no-sol-n} follows directly from
Theorem \ref{no-even-solution}.
\end{remark}

It follows from \cite{EG} that our condition on $n_k$ in Theorem \ref{no-even-solution} is \emph{sharp}. In fact,
recently Eremenko and Gabrielov \cite{EG} studied (\ref{mean}) from
the viewpoint of geometry, i.e. by studying a related problem concerning
spherical quadrilaterals. Among other things, they prove the following result.
\medskip

\noindent \textbf{Theorem A}. \cite[Theorem 1.3]{EG} \textit{Let $n_{k}%
\in \mathbb{Z}_{\geq 0}$ for all $k$ with $\max_{k} n_{k}\geq1$. Then
equation (\ref{mean}) has an even and symmetric solution $u(z)$, i.e.
\begin{equation}
\label{even-symmetric}u(z)=u(-z)\quad \text{and}\quad u(z)=u(\bar{z}),
\end{equation}
on some rectangular torus $E_{\tau}$ (i.e. for some $\tau \in i\mathbb{R}_{>0}%
$) if and only if $(n_{0},n_{1},n_{2},n_{3})$ satisfies either (\ref{c1}) or (\ref{c2}).}
\medskip

Theorem A indicates that our condition on $n_k$ in Theorem \ref{no-even-solution} is {\it sharp}. Furthermore,
Theorem \ref{no-even-solution} improves Theorem A because we remove the
symmetric assumption $u(z)=u(\bar{z})$. We emphasize that this improvement is not trivial at all, because our numerical computation shows that there exist $1<b_1<b_2<\sqrt{3}$ such that for any $\tau=ib$ with $b\in (b_1,b_2)$,
\[\Delta u+ e^u=16\pi \delta_0+16\pi \delta_{\omega_3/2}\;\text{on}\; E_{\tau}\]has \emph{no even and symmetric solutions} but does have \emph{two even solutions}. In view of Theorem
\ref{no-sol-n}, we suspect that the even assumption $u(z)=u(-z)$ is not necessary either, namely we propose the following conjecture.

\begin{conjecture}
Equation (\ref{mean}) has no solution for any $\tau\in i\mathbb{R}_{>0}$ if and only if $(n_{0},n_{1},n_{2},n_{3})$ satisfies neither (\ref{c1}) nor (\ref{c2}).
\end{conjecture}

Differently from Eremenko and Gabrielov's geometric approach \cite{EG}, we
prove Theorem \ref{no-even-solution} from the viewpoint of the integrable
system in the sense that any solution $u$ can be expressed as some holomorphic
data, by which we will connect the curvature equation (\ref{mean}) to the
following generalized Lam\'{e} equation (GLE, a second order linear ODE)
\begin{equation}
y^{\prime \prime}(z)=I(z)y(z)=\bigg[  \sum_{k=0}^{3}n_{k}(n_{k}+1)\wp \left(
z+\tfrac{\omega_{k}}{2};\tau \right)  +E\bigg]  y(z). \label{GLE-0}%
\end{equation}
See Section \ref{PDE-unitary} for details. Here $\wp(z)=\wp(z;\tau)$ is the
Weierstrass elliptic function with periods $\omega_{1}=1$ and $\omega_{2}%
=\tau$, defined by
\begin{equation}
\wp(z;\tau):=\frac{1}{z^{2}}+\sum_{\omega \in \Lambda_{\tau}\setminus
\{0\}}\left(  \frac{1}{(z-\omega)^{2}}-\frac{1}{\omega^{2}}\right)  .
\label{wp}%
\end{equation}
Note that GLE (\ref{GLE-0}) becomes the classical Lam\'{e} equation when three
$n_{k}$'s vanish, such as $n_{1}=n_{2}=n_{3}=0$. GLE (\ref{GLE-0}) is the
elliptic form of Heun's equation and the potential
\begin{equation}\label{Tre}
q^{(n_{0},n_{1},n_{2},n_{3})}(z):=-\sum_{k=0}^{3}n_{k}(n_{k}+1)\wp \left(
z+\tfrac{\omega_{k}}{2};\tau \right)
\end{equation}
is the so-called \emph{Treibich-Verdier potential }(\cite{TV}), which is known
as an algebro-geometric finite-gap potential associated with the stationary
KdV hierarchy. We refer the readers to \cite{GW1,Tak1,Tak2,Tak3,Tak4,Tak5,TV}
and references therein for historical reviews and subsequent developments.

In Section \ref{PDE-unitary}, we will prove that \textit{once
(\ref{mean}) has an even solution, then there exists $E\in \mathbb{C}$ such
that the monodromy group of GLE (\ref{GLE-0}) is conjugate to a subgroup of
$SU(2)$, i.e. the monodromy of GLE (\ref{GLE-0}) is unitary}.
Therefore, to prove Theorem \ref{no-even-solution}, it suffices for us to prove the following result, which is also
interesting from the viewpoint of the monodromy theory of linear ODEs.

\begin{theorem}
\label{no-unitary} Let $n_{k}\in \mathbb{Z}_{\geq 0}$ for all $k$ with
$\max_{k} n_{k}\geq1$ and $\tau \in i\mathbb{R}_{>0}$. If $(n_{0},n_{1},n_{2},n_{3})$ satisfies neither (\ref{c1}) nor (\ref{c2}), then the
monodromy of GLE (\ref{GLE-0}) can not be unitary for any $E\in \mathbb{C}$.
\end{theorem}

It is well known (cf. \cite{GW1,Tak1}) that there associates a so-called {\it spectral polynomial} $Q^{(n_0,n_1,n_2,n_3)}(E;\tau)$ for the Treibich-Verdier potential (\ref{Tre}); see Section \ref{general-lame} for a brief review. In this paper, as applications of Theorem \ref{no-even-solution} and Theorem A, we have the following surprising result.

\begin{theorem}[=Corollary \ref{sharp-realroot1}]\label{sharp-realroot11}
Let $n_{k}\in \mathbb{Z}_{\geq 0}$ for all $k$ with $\max_k n_{k}\geq1$. Then all the zeros of $Q^{(n_{0},n_{1},n_{2},n_{3})}(\cdot;\tau)$ are real
and distinct for each $\tau \in i\mathbb{R}_{>0}$ if and only if $(n_{0}%
,n_{1},n_{2}$, $n_{3})$ satisfies neither (\ref{c1}) nor (\ref{c2}).
\end{theorem}

This paper is organized as follows. In Section 2, we establish the connection between the curvature
equation (\ref{mean}) and GLE (\ref{GLE-0}).
 Theorem \ref{no-unitary} will be proved by
applying the spectral theory of Hill's equations with complex-valued potentials
\cite{GW}; see Section \ref{general-lame}, where we will
prove a more general result (see Theorem \ref{thm4}) which contains Theorem
\ref{no-unitary} as a special case. In Section \ref{int-sys}, we apply Theorem \ref{thm4} to prove a more general
result for equation (\ref{mean}) (see Theorem \ref{thm6}) which contains
Theorem \ref{no-even-solution} as a special case. In Section \ref{real-distinct}, we prove the sufficient part of Theorem \ref{sharp-realroot11}. The necessary part will be a consequence of our results as explained in Section \ref{int-sys}. In Section \ref{premodularform}, we will apply Theorem \ref{no-sol-n} to study pre-modular forms introduced in \cite{CLW2}. In Appendix A, we briefly review the spectral theory of Hill's equation from \cite{GW} that are needed in Section \ref{general-lame}.

\section{From PDE to ODE with unitary monodromy}

\label{PDE-unitary}

The purpose of this section is to establish the connection between the curvature equation (\ref{mean}) and
GLE (\ref{GLE-0}) from the viewpoint of the integrable system. See \cite{CLW}
for a complete discussion for the special case $n_{1}=n_{2}=n_{3}=0$, i.e.
equation (\ref{eq1}). Our argument of proving the unitary monodromy is different from \cite{CLW} and works for the general case
\begin{equation}
\Delta u+e^{u}=4\pi \sum_{k=1}^{m}\alpha_{k}\delta_{a_k}\text{ on
}E_{\tau}, \label{mean-multiple}%
\end{equation}
where $a_k$'s are $m$ different points on $E_{\tau}$, $\alpha_{k}>-1$ for all $k$ and $\sum \alpha_{k}>0$.

The Liouville theorem says that for any solution $u(z)$ of (\ref{mean-multiple}), there
is a local meromorphic function $f(z)$ away from $\{a_k\}$'s such that%
\begin{equation}
u(z)=\log \frac{8|f^{\prime}(z)|^{2}}{(1+|f(z)|^{2})^{2}}. \label{502}%
\end{equation}
We remark that the classical Liouville theorem holds only for the case when
the domain is simply connected and the equation does not include any
singularity. For our present case, see \cite{CLW} for a proof.

This $f(z)$ is called a developing map. By differentiating (\ref{502}), we
have
\begin{equation}
u_{zz}-\frac{1}{2}u_{z}^{2}= \{ f;z \}:=\left(  \frac{f^{\prime \prime}%
}{f^{\prime}}\right)  ^{\prime}-\frac{1}{2}\left(  \frac{f^{\prime \prime}%
}{f^{\prime}}\right)  ^{2}. \label{new22}%
\end{equation}
Conventionally, the RHS of this identity is called the Schwarzian derivative
of $f(z)$, denoted by $\{ f;z \}$. Note that outside
the singularities $\{ a_k\,|\,1\leq
k\leq m\}+\Lambda_{\tau}$,
\[
\left(  u_{zz}-\tfrac{1}{2}u_{z}^{2}\right)  _{\bar{z}}=\left(  u_{z\bar{z}%
}\right)  _{z}-u_{z}u_{z\bar{z}}=-\tfrac{1}{4}\left(  e^{u}\right)
_{z}+\tfrac{1}{4}e^{u}u_{z}=0.
\]
Furthermore, using the local behavior of $u(z)$ at $a_k$: $u(z)=2\alpha_k\ln|z-a_k|+O(1)$,
we see that $u_{zz}-\frac{1}{2}u_{z}^{2}$ has at most double poles at each $a_k$. In conclusion, $u_{zz}-\frac{1}{2}u_{z}^{2}$ is an \emph{elliptic function} with at most \emph{double poles} at $\{ a_k\,|\,1\leq
k\leq m\}+\Lambda_{\tau}$. Denote
\begin{equation}\label{Iz}I(z):=-\tfrac{1}{2}(u_{zz}-\tfrac{1}{2}u_{z}^{2}),\end{equation}
and consider the
Fuchsian type linear differential equation
\begin{equation}\label{ODE}
y''(z)=I(z)y(z).
\end{equation}
Since $\{f;z\}=-2I(z)$, a classical result says that there exist linearly independent solutions $y_1(z), y_2(z)$ such that
\begin{equation}\label{fy}f(z)=\frac{y_1(z)}{y_2(z)}.\end{equation}

On the other hand,
the monodromy representation of
(\ref{ODE}) is a group homomorphism $\rho(\cdot):\pi_{1}(E_{\tau
}\setminus (\{ a_k\,|\,1\leq
k\leq m\}+\Lambda_{\tau}))\rightarrow SL(2,\mathbb{C})$. In general, the monodromy of such linear differential equation could be very complicate. The main result of this section is the following.

\begin{theorem}\label{general-unitary} If equation (\ref{ODE}) comes from a solution $u(z)$ of (\ref{mean-multiple}) via (\ref{Iz}), then the monodromy group with respect to the linearly independent solutions $(y_1(z), y_2(z))$ is contain in $SU(2)$, i.e. the monodromy is unitary.
\end{theorem}

\begin{proof}
Define the Wronskian
\[W=y_{1}'(z)y_2(z)-y_1(z)y_2'(z).\]
Then $W$ is a nonzero constant. By inserting (\ref{fy}) into (\ref{502}), a direct computation leads to
\[2\sqrt{2}W e^{-\frac{1}{2}u(z)}=|y_1(z)|^2+|y_2(z)|^2.\]
Since $u(z)$ is single-valued and doubly periodic, we immediately see that the monodromy group with respect to $(y_1(z), y_2(z))$ is contained in $SU(2)$, namely the monodromy is unitary.
\end{proof}

Now we apply Theorem \ref{general-unitary} to the curvature equation (\ref{mean}). Suppose (\ref{mean}) has an \emph{even} solution $u(z)$, i.e. $u(z)=u(-z)$. Then the previous argument shows that $u_{zz}-\frac{1}{2}u_{z}^{2}$ is an \textit{even elliptic function}
with singularities only at $E_{\tau}[2]:=\{\frac{\omega_{k}}{2}\,|\,0\leq k\leq 3\}+\Lambda_{\tau}$. By using the local
behaviors of $u(z)$ near $\frac{\omega_{k}}{2}$: $u(z)=4n_{k}\log
|z-\frac{\omega_{k}}{2}|+O(1)$, it is easy to prove that
\begin{equation}
u_{zz}-\frac{1}{2}u_{z}^{2}=-2\bigg[  \sum_{k=0}^{3}n_{k}(n_{k}+1)\wp
(z+\tfrac{\omega_{k}}{2};\tau)+E\bigg]  =:-2I(z), \label{cc-1}%
\end{equation}
where $E$ is some constant, because due to the evenness, $u_{zz}-\frac{1}%
{2}u_{z}^{2}$ has no residues at $z\in E_{\tau}[2]$. Therefore, (\ref{ODE}) becomes the
following GLE
\begin{equation}
y^{\prime \prime}(z)=I(z)y(z)=\bigg[  \sum_{k=0}^{3}n_{k}(n_{k}+1)\wp
(z+\tfrac{\omega_{k}}{2};\tau)+E\bigg]  y(z). \label{GLE-2}%
\end{equation}
We call that \textit{GLE (\ref{GLE-2}) comes from the curvature equation
(\ref{mean}) if the potential $I(z)$ is given by an even solution $u(z)$ of
(\ref{mean}) via (\ref{cc-1})}. Since $n_k\in\mathbb{Z}_{\geq 0}$, we will see in Section \ref{general-lame} that the local monodromy matrix of GLE
(\ref{GLE-2}) at $\frac{\omega_{k}}{2}$ is $I_{2}$, so the developing map
$f(z)=y_{1}(z)/y_{2}(z)$ is single-valued near each $\frac{\omega_{k}}{2}$ and
then can be extended to be an entire meromorphic function in $\mathbb{C}$.
Applying Theorem \ref{general-unitary}, we have
the following result.

\begin{theorem}
\label{thm5} If the curvature equation (\ref{mean}) has an even solution,
then there exists $E\in \mathbb{C}$ such that the monodromy representation of
GLE (\ref{GLE-2}) is unitary.
\end{theorem}

\begin{remark}
\label{rm} Actually the converse statement of Theorem \ref{thm5} is also true,
i.e. if there exists $E\in \mathbb{C}$ such that the monodromy representation
of GLE (\ref{GLE-2}) is unitary, then (\ref{mean}) has even solutions. Since
this statement is not needed in proving the results of this paper and its
proof is much more delicate and longer than that of Theorem \ref{thm5}, we
would like to postpone it in a future work.
\end{remark}

\section{GLE and finite-gap potential}

\label{general-lame}

The purpose of this section is to prove a more general result (see Theorem
\ref{thm4} below) which contains Theorem \ref{no-unitary} as a consequence.
To state Theorem
\ref{thm4}, we need to recall some basic facts about the monodromy theory of the generalized
Lam\'{e} equation (GLE)
\begin{equation}
y^{\prime \prime}(z)=I(z;E)y(z)=\bigg[  \sum_{k=0}^{3}n_{k}(n_{k}+1)\wp \left(
z+\tfrac{\omega_{k}}{2};\tau \right)  +E\bigg]  y(z), \label{GLE-1}%
\end{equation}
where $n_{k}\in \mathbb{Z}_{\geq 0}$ for all $k$ and $\max_{k}n_{k}\geq
1$.
Since the local exponents of GLE (\ref{GLE-1}) at $\frac{\omega_{k}}{2}$
are $-n_{k}$ and $n_{k}+1$ and the potential $I(\cdot;E)$ is even elliptic, it
is easy to prove (cf. \cite{GW1} or \cite[Proposition 3.4]{Tak1}) that the
local monodromy matrix of GLE (\ref{GLE-1}) at $\frac{\omega_{k}}{2}$ is the
unit matrix $I_{2}$. Therefore, the monodromy representation of GLE
(\ref{GLE-1}) is a group homomorphism $\rho(\cdot;E):\pi_{1}(E_{\tau
})\rightarrow SL(2,\mathbb{C})$. Let $\ell_{j}\in \pi_{1}(E_{\tau})$, $j=1,2$,
be the two fundamental cycles of $E_{\tau}$ such that $\ell_{1}\ell_{2}%
\ell_{1}^{-1}\ell_{2}^{-1}=Id$. Then%
\begin{equation}
\rho(\ell_{1};E)\rho(\ell_{2};E)=\rho(\ell_{2};E)\rho(\ell_{1};E), \label{ab}%
\end{equation}
where $\rho(\ell_{j};E)$ denotes the monodromy matrix of GLE (\ref{GLE-1})
with respect to any pair of linearly independent solutions. That is, the
monodromy representation of GLE (\ref{GLE-1}) is always \textit{abelian} and
hence \textit{reducible}. Consequently, there exists a solution $y_{1}%
(z)=y_{1}(z;E)$ of GLE (\ref{GLE-1}) such that $y_{1}(z;E)$ is a common
eigenfunction of $\rho(\ell_{1};E)$ and $\rho(\ell_{2};E)$:
\begin{equation}
y_{1}(z+1;E)=e^{\pi i\theta_{1}(E)}y_{1}(z;E),\quad y_{1}(z+\tau;E)=e^{\pi
i\theta_{2}(E)}y_{1}(z;E), \label{mono1}%
\end{equation}
where $\theta_{j}(E)\in \mathbb{C}$ are some constants, i.e. $y_{1}(z;E)$ is
elliptic of the second kind. Since $I(-z;E)=I(z;E)$ implies that
$y_{2}(z)=y_{2}(z;E):=y_{1}(-z;E)$ is also a solution of the same GLE
(\ref{GLE-1}) and also a common eigenfunction of $\rho(\ell_{1};E)$ and
$\rho(\ell_{2};E)$:
\begin{equation}
y_{2}(z+1;E)=e^{-\pi i\theta_{1}(E)}y_{2}(z;E),\quad y_{2}(z+\tau;E)=e^{-\pi
i\theta_{2}(E)}y_{2}(z;E), \label{mono2}%
\end{equation}
we conclude that $\Phi(z;E):=y_{1}(z;E)y_{2}(z;E)=y_{1}(z;E)y_{1}(-z;E)$ is an
\textit{even elliptic} solution of the second symmetric product equation of
GLE (\ref{GLE-1}):
\begin{equation}
\Phi^{\prime \prime \prime}(z)-4I(z;E)\Phi^{\prime}(z)-2I^{\prime}%
(z;E)\Phi(z)=0. \label{secondsymm}%
\end{equation}
On the other hand, it was proved in \cite[Proposition 2.9]{Tak-MathZ} that
\textit{the dimension of the even elliptic solutions of equation
(\ref{secondsymm}) is $1$}. Together this with \cite[Proposition 3.5]{Tak1},
we immediately obtain

\begin{lemma}
\label{lemma1} Up to a multiplication, $\Phi(z;E)=y_{1}(z;E)y_{1}(-z;E)$ have
the following expression:
\begin{equation}
\Phi(z;E)=c_{0}(E)+\sum_{k=0}^{3}\sum_{j=0}^{n_{k}-1}b_{j}^{(k)}%
(E)\wp(z+\tfrac{\omega_{k}}{2};\tau)^{n_{k}-j},
\end{equation}
where the coefficients $c_{0}(E)$ and $b_{j}^{(k)}(E)$ are polynomials in $E$,
they do not have common divisors, and $c_{0}(E)$ is monic. Set $g=\deg
c_{0}(E)$, then $\deg b_{j}^{(k)}(E)<g$ for all $k$ and $j$.
\end{lemma}

Define
\[
W(E):=y_{1}^{\prime}(z;E)y_{2}(z;E)-y_{1}(z;E)y_{2}^{\prime}(z;E)
\]
to be the Wronskian of $y_{1}(z;E)$ and $y_{2}(z;E)$. Clearly $W(E)$ is a
constant independent of $z$. It is easy to see that
\[
\frac{y_{1}^{\prime}(z;E)}{y_{1}(z;E)}=\frac{\Phi^{\prime}(z;E)+W(E)}%
{2\Phi(z;E)},\quad \frac{y_{2}^{\prime}(z;E)}{y_{2}(z;E)}=\frac{\Phi^{\prime
}(z;E)-W(E)}{2\Phi(z;E)},
\]
which implies
\begin{align*}
\frac{\Phi^{\prime \prime}(z;E)}{2\Phi(z;E)}-\frac{\Phi^{\prime}(z;E)+W(E)}%
{2\Phi(z;E)^{2}}\Phi^{\prime}(z;E)  &  =\left(  \frac{y_{1}^{\prime}%
(z;E)}{y_{1}(z;E)}\right)  ^{\prime}\\
&  =I(z;E)-\left(  \frac{\Phi^{\prime}(z;E)+W(E)}{2\Phi(z;E)}\right)  ^{2},
\end{align*}
and
\[
\frac{\Phi^{\prime \prime}(z;E)}{2\Phi(z;E)}-\frac{\Phi^{\prime}(z;E)-W(E)}%
{2\Phi(z;E)^{2}}\Phi^{\prime}(z;E)=I(z;E)-\left(  \frac{\Phi^{\prime
}(z;E)-W(E)}{2\Phi(z;E)}\right)  ^{2}.
\]
From here we immediately obtain
\begin{equation}
\frac{W(E)^{2}}{4}=I(z;E)\Phi(z;E)^{2}+\frac{\Phi^{\prime2}}{4}-\frac
{\Phi(z;E)\Phi^{\prime \prime}(z;E)}{2}. \label{Wron}%
\end{equation}
Remark that (\ref{Wron}) is well known (cf. \cite{GW1,Tak1}) and the fact that
the RHS of (\ref{Wron}) is independent of $z$ can be also seen from equation
(\ref{secondsymm}). Define
\begin{align}
Q(E)=  &  Q(E;\tau)=Q^{(n_{0},n_{1},n_{2},n_{3})}(E;\tau)
\nonumber \label{poly}\\
:=  &  I(z;E)\Phi(z;E)^{2}+\frac{\Phi^{\prime2}}{4}-\frac{\Phi(z;E)\Phi
^{\prime \prime}(z;E)}{2}.
\end{align}
Then it follows from the expression of $I(z;E)$ and Lemma \ref{lemma1} that
$Q(E)$ is \textit{a monic polynomial of degree $2g+1$}. This polynomial $Q(E)$
is known as \textit{the special polynomial of the Treibich-Verdier potential
}(cf. \cite{GW1,Tak1}) and will play a crucial role in this paper. The number
$g$, i.e. the arithmetic genus of the hyperelliptic curve $F^{2}=Q(E)$, was
computed in \cite{GW1, Tak5}: Let $m_{k}$ be the rearrangement of $n_{k}$ such
that $m_{0}\geq m_{1}\geq m_{2}\geq m_{3}$, then
\begin{equation}
g=%
\begin{cases}
m_{0}, & \text{if $\sum m_{k}$ is even and $m_{0}+m_{3}\geq m_{1}+m_{2}$};\\
\frac{m_{0}+m_{1}+m_{2}-m_{3}}{2}, & \text{if $\sum m_{k}$ is even and
$m_{0}+m_{3}<m_{1}+m_{2}$};\\
m_{0}, & \text{if $\sum m_{k}$ is odd and $m_{0}>m_{1}+m_{2}+m_{3}$};\\
\frac{m_{0}+m_{1}+m_{2}+m_{3}+1}{2}, & \text{if $\sum m_{k}$ is odd and
$m_{0}\leq m_{1}+m_{2}+m_{3}$}.
\end{cases}
\label{genus}%
\end{equation}
Furthermore, it is known (cf. \cite{GW1,Tak5}) that the roots of $Q(\cdot
;\tau)=0$ are \textit{distinct} for generic $\tau \in \mathbb{H}$.

We summarize the above argument in the following

\begin{lemma}
\label{lemma2} The Wronskian $W(E)$ of $y_{1}(z;E), y_{2}(z;E)=y_{1}(-z;E)$
satisfies
\[
(W(E)/2)^{2}=Q(E),
\]
where $Q(E)$ is a monic polynomial of degree $2g+1$, defined by (\ref{poly})
with $g$ given by (\ref{genus}). Furthermore,

\begin{itemize}
\item[(1)] if $Q(E)\neq0$, then the monodromy group of GLE (\ref{GLE-1}) with
respect to $(y_{1}(z;E), y_{2}(z;E))$ is generated by
\begin{equation}
\rho(\ell_{1}; E)=%
\begin{pmatrix}
e^{\pi i\theta_{1}(E)} & 0\\
0 & e^{-\pi i\theta_{1}(E)}%
\end{pmatrix}
,\; \; \rho(\ell_{2}; E)=%
\begin{pmatrix}
e^{\pi i\theta_{2}(E)} & 0\\
0 & e^{-\pi i\theta_{2}(E)}%
\end{pmatrix}
, \label{cc-30}%
\end{equation}
where $(\theta_{1}(E),\theta_{2}(E))$ is seen in (\ref{mono1})-(\ref{mono2}).
Besides, $(\theta_{1}(E),\theta_{2}(E))\notin \mathbb{Z}^{2}$.

\item[(2)] if $Q(E)=0$, then the dimension of the common eigenfunctions is $1$
and $(\theta_{1}(E),\theta_{2}(E))\in \mathbb{Z}^{2}$.

\item[(3)] the roots of $Q(\cdot;\tau)=0$ are distinct for generic $\tau
\in \mathbb{H}$.
\end{itemize}
\end{lemma}

\begin{proof}
It suffices for us to prove the assertions (1)-(2).

(1) Suppose $Q(E)\neq0$, then $y_{1}(z;E)$ and $y_{2}(z;E)$ are linearly
independent and hence (\ref{cc-30}) follows from (\ref{mono1})-(\ref{mono2}).
Assume by contradiction that $(\theta_{1}(E),\theta_{2}(E))\in \mathbb{Z}^{2}$,
i.e. $e^{\pi i \theta_{j}(E)}=e^{-\pi i \theta_{j}(E)}\in \{ \pm1\}$ for
$j=1,2$. Let $y(z)=y_{1}(z;E)+y_{2}(z;E)$ be a solution of GLE (\ref{GLE-1}). Then it
follows from (\ref{mono1})-(\ref{mono2}) that $y(z)y(-z)$ is an even elliptic
solution of (\ref{secondsymm}). Again by
\cite[Proposition 2.9]{Tak-MathZ} that the dimension of even elliptic solutions of
(\ref{secondsymm}) is $1$, we conclude
$y(z)y(-z)=cy_{1}(z;E)y_{2}(z;E)$ for some constant $c\neq0$, i.e. either
$y(z)=c_{1}y_{1}(z;E)$ or $y(z)=c_{1}y_{2}(z;E)$ with some constant
$c_{1}\neq0$, a
contradiction with $y(z)=y_{1}(z;E)+y_{2}(z;E)$. This proves $(\theta_{1}(E), \theta_{2}(E))\notin
\mathbb{Z}^{2}$.

(2) Suppose $Q(E)=0$, then $y_{1}(z;E)$ and $y_{2}(z;E)=y_{1}(-z;E)$ are
linearly dependent and hence (\ref{mono1})-(\ref{mono2}) imply $e^{\pi i
\theta_{j}(E)}=e^{-\pi i \theta_{j}(E)}$ for $j=1,2$, i.e. $(\theta
_{1}(E),\theta_{2}(E))\in \mathbb{Z}^{2}$. If there exists another common
eigenfunction $y(z)$ which is linearly independent with $y_{1}(z;E)$, then the
same argument as (1) shows $y(z)y(-z)=cy_{1}(z;E)^{2}$, clearly a
contradiction. This proves that the dimension of the common eigenfunctions is
$1$, i.e. the monodormy matrix $\rho(\ell_{1}; E)$ and $\rho(\ell_{2}; E)$ can
not be diagonized simultaneously.
\end{proof}

By Lemma \ref{lemma2}, it is easy to see that the following corollary holds.

\begin{corollary}
\label{coro} The monodromy representation of GLE (\ref{GLE-1}) is
\textit{unitary}, i.e. the monodromy group is contained in $SU(2)$ up to a
common conjugation, if and only if $Q(E)\neq0$ and $(\theta_{1}(E), \theta
_{2}(E))\in \mathbb{R}^{2}\setminus \mathbb{Z}^{2}$.
\end{corollary}

The main result of this section is as follows, which is interesting from the
viewpoint of the monodromy theory of linear ODEs. Theorem \ref{no-unitary}
will be a consequence of this result.

\begin{theorem}
\label{thm4} Let $\tau \in i\mathbb{R}_{>0}$ and $n_{k}\in \mathbb{Z}_{\geq 0}$
for all $k$ with $\max_k n_k\geq1$. Suppose that all zeros of $Q^{(n_{0}, n_{1},
n_{2}, n_{3})}(\cdot;\tau)$ are real and distinct. Then the monodromy
representation of GLE (\ref{GLE-1}) can not be unitary for any $E\in
\mathbb{C}$.
\end{theorem}

The rest of this section is devoted to the proof of Theorem \ref{thm4}.
Recalling Lemma \ref{lemma2}, we have
\begin{equation}
\label{trace1}\operatorname{tr} \rho(\ell_{1}; E)=e^{\pi i \theta_{1}%
(E)}+e^{-\pi i \theta_{1}(E)}=2\cos(\pi \theta_{1}(E))\; \; \text{for all $E$}.
\end{equation}
Remark that $\operatorname{tr} \rho(\ell_{1}; E)$ is independent of the choice
of linearly independent solutions. Define
\begin{equation}
\label{tildes}\tilde{\mathcal{S}}=\tilde{\mathcal{S}}^{(n_{0},n_{1}%
,n_{2},n_{3})}(\tau):=\{E\in \mathbb{C}\,|\, -2\leq \operatorname{tr} \rho
(\ell_{1}; E)\leq2\}.
\end{equation}

\begin{lemma}
\label{lemma4} Let $\tau \in i\mathbb{R}_{>0}$. Then the set $\tilde
{\mathcal{S}}$ is symmetric with respect to the real line $\mathbb{R}$.
\end{lemma}

\begin{proof}
It suffices to prove $\bar{E}\in \tilde{\mathcal{S}}$ if $E\in \tilde
{\mathcal{S}}$. Suppose $E\in \tilde{\mathcal{S}}$, then $\theta_{1}%
(E)\in \mathbb{R}$. Since $\tau \in i\mathbb{R}_{>0}$, it follows from the
expression (\ref{wp}) of $\wp(z;\tau)$ that $\overline{\wp(z;\tau)}=\wp
(\bar{z};\tau)$. Note that $\wp(z+\frac{\overline{\omega_{k}}}{2};\tau
)=\wp(z+\frac{\omega_{k}}{2};\tau)$. Then it is easy to see that $y(z)$ is a
solution of GLE (\ref{GLE-1}) if and only if $\tilde{y}(z):=\overline
{y(\bar{z})}$ is a solution of GLE
\begin{equation}
\tilde{y}^{\prime \prime}(z)=I(z;\bar{E})\tilde{y}(z)=\bigg[  \sum_{k=0}%
^{3}n_{k}(n_{k}+1)\wp \left(  z+\tfrac{\omega_{k}}{2};\tau \right)  +\bar
{E}\bigg]  \tilde{y}(z). \label{GLE-6}%
\end{equation}
Let $y_{1}(z;E)$ be the common eigenfunction of the monodromy matrices of GLE
(\ref{GLE-1}) such that (\ref{mono1}) holds, then $\tilde{y}_{1}%
(z):=\overline{y_{1}(\bar{z};E)}$ satisfies
\[
\tilde{y}_{1}(z+1)=e^{-\pi i\theta_{1}(E)}\tilde{y}_{1}(z),\; \; \tilde{y}%
_{1}(z+\tau)=e^{\pi i\overline{\theta_{2}(E)}}\tilde{y}_{1}(z),
\]
where $\theta_{1}(E)\in \mathbb{R}$ is used. Therefore, $\tilde{y}_{1}(z)$ is a
common eigenfunction of the monodromy matrices of GLE (\ref{GLE-6}) and so
\[
\operatorname{tr}\rho(\ell_{1};\bar{E})=e^{-\pi i\theta_{1}(E)}+e^{\pi
i\theta_{1}(E)}=\operatorname{tr}\rho(\ell_{1};E)\in \lbrack-2,2].
\]
This proves $\bar{E}\in \tilde{\mathcal{S}}$.
\end{proof}

The following lemma is the only place of this paper where we need to use some notions and classical results about the spectral theory of Hill's equations that will be recalled in Appendix \ref{Floquet-theory}.
\begin{lemma}
\label{lemma5} Let $\tau \in i\mathbb{R}_{>0}$ and $n_{k}\in \mathbb{Z}_{\geq 0}$ for all $k$ with $\max_k n_{k}\geq1$. Suppose $Q(E;\tau)=Q^{(n_{0}, n_{1},
n_{2}, n_{3})}(E;\tau)$ has $2g+1$ real distinct zeros, denoted by
$E_{2g}<E_{2g-1}<\cdots<E_{1}<E_{0}$. Then
\begin{equation}
\label{spectrum1}\tilde{\mathcal{S}}^{(n_{0},n_{1},n_{2},n_{3})}%
(\tau)=(-\infty, E_{2g}]\cup[E_{2g-1}, E_{2g-2}]\cup \cdots \cup[E_{1}, E_{0}].
\end{equation}

\end{lemma}

\begin{proof}
To avoid the singularities on $\mathbb{R}$,
we let $z=x+\frac{\tau}{4}$ with $x\in \mathbb{R}$ and $w(x):=y(z)=y(x+\frac{\tau
}{4})$ in GLE (\ref{GLE-1}). Then $w(x)$ satisfies the following Hill's
equation
\begin{equation}
w^{\prime \prime}(x)-\bigg[  \sum_{k=0}^{3}n_{k}(n_{k}+1)\wp \left(
x+\tfrac{\tau}{4}+\tfrac{\omega_{k}}{2};\tau \right)  \bigg]
w(x)=Ew(x),\;x\in \mathbb{R}, \label{GLE-1-1}%
\end{equation}
with the potential $q(x):=-\sum_{k=0}^{3}n_{k}(n_{k}+1)\wp \left(
x+\tfrac{\tau}{4}+\tfrac{\omega_{k}}{2};\tau \right)  $ being {\it smooth} on
$\mathbb{R}$ with period $\Omega=1$.
Let $w_{1}(x)$, $w_{2}(x)$ be any two linearly independent solutions of (\ref{GLE-1-1}). Then so do $w_{1}(x+\Omega)$, $w_{2}(x+\Omega)$ and
hence there is a monodromy matrix $M(E)\in SL(2,\mathbb{C})$ such that
\[
(w_{1}(x+\Omega),w_{2}(x+\Omega))=(w_{1}(x), w_{2}(x))M(E).
\]
As in Appendix \ref{Floquet-theory}, we define the \emph{Hill's discriminant} $\Delta(E)$ by
\begin{equation}
\label{trace11}\Delta(E):=\text{tr}M(E),
\end{equation}
which is independent of the choice of solutions. This $\Delta(E)$ is  an entire function and plays a
fundamental role since it encodes all the spectrum
information of the associated operator. Indeed, we define
\[
\mathcal{S}:=\Delta^{-1}([-2,2])=\{E\in \mathbb{C}\,|\, -2\leq \Delta(E)\leq2\}
\]
to be \textit{the conditional stability set} of the operator $L=\frac{d^{2}%
}{dx^{2}}+q(x)$. Since $q(x)$ is continuous, it was proved in \cite{Rofe-Beketov} that \textit{this $\mathcal{S}$
coincides with the spectrum $\sigma(H)$ of the associated linear operator $H$
in $L^{2}(\mathbb{R}, \mathbb{C})$} (i.e. $H$ is defined as $Hf=Lf$, $f\in
H^{2,2}(\mathbb{R},\mathbb{C})$).

Observe that

\begin{itemize}
\item[$(\star)$] \textit{if $(y_{1}(z),y_{2}(z))$ is a pair of linearly
independent solutions of GLE (\ref{GLE-1}), then $(w_{1}(x), w_{2}%
(x)):=(y_{1}(x+\frac{\tau}{4}),y_{2}(x+\frac{\tau}{4}))$ is a pair of linearly
independent solutions of equation (\ref{GLE-1-1})}.
\end{itemize}

\noindent Thus, the monodromy matrix $\rho(\ell_{1};E)$ is also the monodromy
matrix of equation (\ref{GLE-1-1}), which gives
\[
\Delta(E)=\operatorname{tr}\rho(\ell_{1};E)=2\cos(\pi \theta_{1}(E))
\]
and so we obtain the following important identity
\begin{equation}
\sigma(H)=\mathcal{S}=\{E\in \mathbb{C}\,|\,-2\leq \Delta(E)\leq2\}=\tilde{\mathcal{S}}.
\label{spectrum2}%
\end{equation}

On the other hand,  Lemma \ref{lemma2}-(1) and $(\star)$ imply that if $Q(E)\neq 0$, then $w_j(x):=y_j(x+\frac{\tau}{4};E)$, $j=1,2$, are linearly independent Floquet solutions of (\ref{GLE-1-1}). So we can apply Theorem B in Appendix \ref{Floquet-theory} to (\ref{GLE-1-1}). In
particular, $(\star)$ infers that the polynomial $R_{2g+1}(E)$ in Theorem
B-(ii) is precisely $Q(E)$; see also \cite{GW1}. Together with our
assumption, we obtain
\[
R_{2g+1}(E)=Q(E)=\prod_{j=0}^{2g}(E-E_{j}).
\]
Then it follows from (\ref{hyper}) that
\begin{equation}\label{d-odd}
d(E_{j}):=\operatorname{ord}_{E_{j}}(\Delta(\cdot)^{2}-4)=1+2p_{i}(E_{j})\;\;\text{is \textit{odd} for all $j\in \lbrack0,2g]$.}
\end{equation}
By $\sigma(H)=\{E|-2\leq \Delta(E)\leq2\}$, it is easy to prove
(see e.g. the proof of Theorem B-(iii) in \cite{GW}) that there are $d(E_{j})$
semi-arcs of the spectrum $\sigma(H)$ meeting at $E_{j}$. On the other hand,
Theorem B-(iii) says that: The spectrum $\sigma(H)=\mathcal{S}$ consists of
finitely many bounded spectral arcs $\sigma_{k}$, $1\leq k\leq \tilde{g}$ for
some $\tilde{g}\leq g$ and {\it one} semi-infinite arc $\sigma_{\infty}$ which tends
to $-\infty+\langle q\rangle$, i.e.
\[
\sigma(H)=\mathcal{S}=\sigma_{\infty}\cup \cup_{k=1}^{\tilde{g}}\sigma_{k}.
\]
Furthermore, the set of the finite end points of such arcs is precisely $\{E_{j}
\}_{j=0}^{2g}$ because of (\ref{d-odd}). Together these with the following three facts:

\begin{itemize}
\item[(a)] Our assumption gives $E_{j}\in \mathbb{R}$ and $E_{2g}<E_{2g-1}%
<\cdots<E_{1}<E_{0}$;

\item[(b)] Lemma \ref{lemma4} and (\ref{spectrum2}) imply that the spectrum
$\sigma(H)=\mathcal{S}=\tilde{\mathcal{S}}$ is symmetric with respect to the
real line $\mathbb{R}$;

\item[(c)] A classical result (see e.g. \cite[Theorem 2.2]{GW2}) says that
$\mathbb{C}\setminus \sigma(H)$ is path-connected;
\end{itemize}

\noindent we easily conclude that (i) $\sigma(H)\subset \mathbb{R}$, (ii) $d(E_{j})=1$
(i.e. $\frac{d}{dE}\Delta(E_{j})\neq 0$) for all $j$ and so
\begin{equation}\label{s-finite}
\tilde{\mathcal{S}}=\sigma(H)=(-\infty,E_{2g}]\cup \lbrack E_{2g-1}%
,E_{2g-2}]\cup \cdots \cup \lbrack E_{1},E_{0}].
\end{equation}
Indeed, since (a) says that all finite end points of spectral arcs are on $\mathbb{R}$, the assertion (i) $\sigma(H)\subset \mathbb{R}$ follows immediately from (b)-(c). Consequently, there are at most two semi-arcs of $\sigma(H)$ meeting at each $E_j$. This, together with (\ref{d-odd}), yields the assertion (ii) $d(E_j)=1$ for all $j$, namely there is exactly one semi-arc of $\sigma(H)$ ending at $E_j$, which finally implies (\ref{s-finite}).
The proof is complete.
\end{proof}

On the other hand, we have the following important observation.

\begin{lemma}
\label{lemma3} Fix $n_{k}\in \mathbb{Z}_{\geq 0}$ for all $k$ with $\max_k n_{k}\geq
1$. Let $E_{j}=E_{j}(\tau)$, $j=0,1,\cdots, 2g$, are all roots of
$Q^{(n_{0},n_{1},n_{2},n_{3})}(\cdot;\tau)=0$, i.e.
\[
Q^{(n_{0},n_{1},n_{2},n_{3})}(E;\tau)=\prod_{j=0}^{2g}(E-E_{j}(\tau)).
\]
Then
\begin{equation}
\label{modular1}Q^{(n_{0},n_{2},n_{1},n_{3})}(E;\tfrac{-1}{\tau})=\prod
_{j=0}^{2g}(E-\tau^{2}E_{j}(\tau)).
\end{equation}

\end{lemma}

\begin{proof}
Recall the modular property of $\wp(z;\tau)$:
\[
\wp(z;\tfrac{-1}{\tau})=\tau^{2}\wp(\tau z;\tau),
\]
which gives
\[
\wp(z+\tfrac{1}{2};\tfrac{-1}{\tau})=\tau^{2}\wp(\tau z+\tfrac{\tau}{2}%
;\tau),
\]%
\[
\wp(z+\tfrac{-1}{2\tau};\tfrac{-1}{\tau})=\tau^{2}\wp(\tau z+\tfrac{1}{2}%
;\tau),
\]%
\[
\wp(z+\tfrac{\tau-1}{2\tau};\tfrac{-1}{\tau})=\tau^{2}\wp(\tau z+\tfrac
{1+\tau}{2};\tau).
\]
From here, we immediately see that $y(z)$ is a solution of GLE (\ref{GLE-1})
if and only if $\tilde{y}(z):=y(\tau z)$ is a solution of GLE
\begin{equation}
\tilde{y}^{\prime \prime}(z)=\bigg[  \sum_{k=0}^{3}\tilde{n}_{k}(\tilde{n}%
_{k}+1)\wp \left(  z+\tfrac{\omega_{k}}{2};\tfrac{-1}{\tau}\right)  +\tau
^{2}E\bigg]  \tilde{y}(z) \label{GLE-4}%
\end{equation}
with $(\tilde{n}_{0},\tilde{n}_{1},\tilde{n}_{2},\tilde{n}_{3})=(n_{0}%
,n_{2},n_{1},n_{3})$ (of course, we mean $\omega_{2}=\frac{-1}{\tau}$ and
$\omega_{3}=1+\frac{-1}{\tau}$ in (\ref{GLE-4})).

Let $y_{1}(z;E)$ be the common eigenfunction of the monodromy matrices of GLE
(\ref{GLE-1}) such that (\ref{mono1}) holds, then $\tilde{y}_{1}%
(z;E):=y_{1}(\tau z;E)$ satisfies
\[
\tilde{y}_{1}(z+1;E)=e^{\pi i \theta_{2}(E)}\tilde{y}_{1}(z;E),\; \; \tilde
{y}_{1}(z+\tfrac{-1}{\tau};E)=e^{-\pi i \theta_{1}(E)}\tilde{y}_{1}(z;E),
\]
i.e. $\tilde{y}_{1}(z;E)$ is a common eigenfunction of the monodromy matrices
of GLE (\ref{GLE-4}). When $E=E_{j}(\tau)$ (resp. $E\notin \{E_{j}\}_{j=0}%
^{2g}$), Lemma \ref{lemma2} says that $y_{1}(z;E)$ and $y_{1}(-z;E)$ are
linearly dependent (resp. linearly independent) and so do $\tilde{y}_{1}(z;E)$
and $\tilde{y}_{1}(-z;E)$. This yields that $\tau^{2} E_{j}(\tau)$,
$j=0,1,\cdots,2g$, are all the roots of
\[
Q^{(\tilde{n}_{0},\tilde{n}_{1},\tilde{n}_{2},\tilde{n}_{3})}(\cdot;\tfrac
{-1}{\tau}) =Q^{(n_{0},n_{2},n_{1},n_{3})}(\cdot;\tfrac{-1}{\tau})=0.
\]
Since (\ref{genus}) gives $\deg Q^{(n_{0},n_{2},n_{1},n_{3})}(\cdot;\tfrac
{-1}{\tau})=\deg Q^{(n_{0},n_{1},n_{2},n_{3})}(\cdot;\tau)=2g+1$, and
$E_{j}(\tau)$, $j=0,1,\cdots, 2g$, are all distinct for generic $\tau
\in \mathbb{H}$, it follows that (\ref{modular1}) holds for generic $\tau
\in \mathbb{H}$ and hence for all $\tau \in \mathbb{H}$ by continuity with
respect to $\tau$.
\end{proof}

Now we are in the position to prove Theorem \ref{thm4}.

\begin{proof}
[Proof of Theorem \ref{thm4}]Let $\tau \in i\mathbb{R}_{>0}$ and suppose
$Q^{(n_{0}, n_{1}, n_{2}, n_{3})}(\cdot;\tau)$ has $2g+1$ real distinct zeros,
denoted by $E_{2g}<E_{2g-1}<\cdots<E_{1}<E_{0}$. Then Lemma \ref{lemma5}
gives
\[
\tilde{\mathcal{S}}^{(n_{0},n_{1},n_{2},n_{3})}(\tau)=(-\infty, E_{2g}%
]\cup[E_{2g-1}, E_{2g-2}]\cup \cdots \cup[E_{1}, E_{0}].
\]

Assume by contradiction that the monodromy representation of GLE (\ref{GLE-1})
is unitary for some $E=\hat{E}$. Then Corollary \ref{coro} implies
\begin{equation}
\label{thetar}Q^{(n_{0}, n_{1}, n_{2}, n_{3})}(\hat{E};\tau)\neq
0\quad \text{and}\quad(\theta_{1}(\hat{E}),\theta_{2}(\hat{E}))\in
\mathbb{R}^{2}\setminus \mathbb{Z}^{2}.
\end{equation}
It follows from the definition (\ref{tildes}) of $\tilde{\mathcal{S}}%
^{(n_{0},n_{1},n_{2},n_{3})}(\tau)$ that $\hat{E}\in \tilde{\mathcal{S}%
}^{(n_{0},n_{1},n_{2},n_{3})}(\tau)$ and $E\neq E_{j}$ for all $j$, i.e.
\begin{equation}
\label{ehat}\hat{E}\in(-\infty, E_{2g})\cup(E_{2g-1}, E_{2g-2})\cup \cdots
\cup(E_{1}, E_{0}).
\end{equation}

Note that $\frac{-1}{\tau}\in i\mathbb{R}_{>0}$ and
\[
\tau^{2}E_{0}<\tau^{2}E_{1}<\cdots<\tau^{2}E_{2g-1}<\tau^{2}E_{2g}.
\]
Since Lemma \ref{lemma3} shows that $Q^{(n_{0},n_{2},n_{1},n_{3})}(\cdot
;\frac{-1}{\tau})$ has $2g+1$ real distinct zeros $\{ \tau^{2}E_{j}%
\}_{j=0}^{2g}$, Lemma \ref{lemma5} applies for $\frac{-1}{\tau}$ and
$(n_{0},n_{2},n_{1},n_{3})$ and gives
\[
\tilde{\mathcal{S}}^{(n_{0},n_{2},n_{1},n_{3})}(\tfrac{-1}{\tau}%
)=(-\infty,\tau^{2}E_{0}]\cup \lbrack \tau^{2}E_{1},\tau^{2}E_{2}]\cup \cdots
\cup \lbrack \tau^{2}E_{2g-1},\tau^{2}E_{2g}].
\]
On the other hand, the proof of Lemma \ref{lemma3} shows that $\tilde{y}%
_{1}(z):=y_{1}(\tau z;\hat{E})$ is a common eigenfunction of the monodromy
matrices of GLE
\begin{equation}
\tilde{y}^{\prime \prime}(z)=\bigg[  \sum_{k=0}^{3}\tilde{n}_{k}(\tilde{n}%
_{k}+1)\wp \left(  z+\tfrac{\omega_{k}}{2};\tfrac{-1}{\tau}\right)  +\tau
^{2}\hat{E}\bigg]  \tilde{y}(z) \label{GLE-7}%
\end{equation}
with $(\tilde{n}_{0},\tilde{n}_{1},\tilde{n}_{2},\tilde{n}_{3})=(n_{0}%
,n_{2},n_{1},n_{3})$ and
\[
\tilde{y}_{1}(z+1)=e^{\pi i\theta_{2}(\hat{E})}\tilde{y}_{1}(z).
\]
Consequently, for GLE (\ref{GLE-7}) there holds
\[
\operatorname{tr}\rho(\ell_{1};\tau^{2}\hat{E})=e^{\pi i\theta_{2}(\hat{E}%
)}+e^{-\pi i\theta_{2}(\hat{E})}=2\cos(\pi \theta_{2}(\hat{E}))\in \lbrack-2,2]
\]
by (\ref{thetar}). This implies $\tau^{2}\hat{E}\in \tilde{\mathcal{S}}%
^{(n_{0},n_{2},n_{1},n_{3})}(\tfrac{-1}{\tau})$, i.e.
\[
\tau^{2}\hat{E}\in(-\infty,\tau^{2}E_{0}]\cup \lbrack \tau^{2}E_{1},\tau
^{2}E_{2}]\cup \cdots \cup \lbrack \tau^{2}E_{2g-1},\tau^{2}E_{2g}]
\]
and hence
\[
\hat{E}\in \lbrack E_{2g},E_{2g-1}]\cup \cdots \cup \lbrack E_{2},E_{1}%
]\cup \lbrack E_{0},+\infty),
\]
which is a contradiction with (\ref{ehat}). Therefore, the monodromy
representation of GLE (\ref{GLE-1}) can not be unitary for any $E\in
\mathbb{C}$.
\end{proof}

\section{Application to curvature equation and $Q^{(n_0,n_1,n_2,n_3)}$}

\label{int-sys}

The purpose of this section is apply the previous results to prove Theorems \ref{no-even-solution}, \ref{no-unitary} and \ref{sharp-realroot11}.

By Theorems \ref{thm4} and \ref{thm5}, we immediately obtain the
following general result which contains Theorem \ref{no-even-solution} as a consequence.

\begin{theorem}
\label{thm6} Let $\tau \in i\mathbb{R}_{>0}$ and $n_{k}\in \mathbb{Z}_{\geq 0}$
for all $k$ with $\max_k n_{k}\geq1$. Suppose that all zeros of $Q^{(n_{0}, n_{1},
n_{2}, n_{3})}(\cdot;\tau)$ are real and distinct. Then the curvature
equation (\ref{mean}) on this $E_{\tau}$ has no even solutions.
\end{theorem}

\begin{remark}\label{remark4-2} The converse statement of Theorem \ref{thm6} does not necessarily hold. Here is an example. Define
$e_k=e_k(\tau):=\wp(\tfrac{\omega_k}{2};\tau)$, $k=1,2,3$.
It is well known that
\[e_1(\tau)>e_3(\tau)>e_2(\tau)\quad\text{for } \tau \in i\mathbb{R}_{>0},\]
\[e_1(\tau)\in \mathbb{R},\quad e_2(\tau)=\overline{e_3(\tau)}\notin \mathbb{R}\quad\text{for } \tau \in \tfrac{1}{2}+i\mathbb{R}_{>0}.\]
Now for $\tau \in i\mathbb{R}_{>0}$, we have $\frac{1+\tau}{2}\in \tfrac{1}{2}+i\mathbb{R}_{>0}$ and it is easy to compute that
\[Q^{(1, 0,
0, 1)}(E;\tau)=(E-E_0(\tau))(E-E_1(\tau))(E-E_2(\tau))\]
with $E_0(\tau)=e_1(\frac{1+\tau}{2})-2e_3(\tau)\in \mathbb{R}$, $E_1(\tau)=\overline{E_2(\tau)}=e_2(\frac{1+\tau}{2})-2e_3(\tau)\notin \mathbb{R}$, namely $Q^{(1, 0,
0, 1)}(E;\tau)$ always has two roots in $\mathbb{C}\setminus\mathbb{R}$ for any $\tau\in i\mathbb{R}_{>0}$. On the other hand, it was proved in \cite[Theorem 1.1]{CKL-2} (see also \cite{EG,EG1}) that there exist $0<b_0<1<b_1<\sqrt{3}$ such that
\[\Delta u+ e^u=8\pi \delta_0+ 8\pi \delta_{\omega_3/2}\quad \text{on}\; E_{\tau},\;\tau=ib,\;b>0\]
has no even solutions if and only if $b\in [b_0, b_1]$.
\end{remark}

Remark \ref{remark4-2} indicates that the zeros of $Q^{(n_{0}, n_{1}, n_{2},
n_{3})}(\cdot;\tau)$ are not necessarily real distinct for $\tau \in i\mathbb{R}_{>0}$ without
further conditions on $n_{k}$'s. Naturally we ask: {\it What are the $n_{k}$'s such that $Q^{(n_{0}, n_{1}, n_{2},
n_{3})}(\cdot;\tau)$ has real distinct zeros}? We have the following result on this aspect.

\begin{theorem}\label{sharp-realroot}
Let $n_{k}\in \mathbb{Z}_{\geq 0}$ for all $k$ with $\max_k n_k\ge 1$. If neither
\begin{equation}
\label{c1-11}\frac{n_{1}+n_{2}-n_{0}-n_{3}}{2}\geq1,\quad n_{1}\geq1,\quad
n_{2}\geq1
\end{equation}
nor
\begin{equation}
\label{c2-11}\frac{n_{1}+n_{2}-n_{0}-n_{3}}{2}\leq-1,\quad n_{0}\geq1,\quad
n_{3}\geq1
\end{equation}
hold, then for any $\tau\in i \mathbb{R}_{>0}$, the zeros of $Q^{(n_{0},n_{1}%
,n_{2},n_{3})}(\cdot;\tau)$ are real and distinct.
\end{theorem}

The proof of Theorem \ref{sharp-realroot} is long and will be postponed in Section \ref{real-distinct}. We will see from Corollary \ref{sharp-realroot1} that our condition on $n_k$ in Theorem \ref{sharp-realroot} is sharp.
Now we can prove Theorems \ref{no-even-solution} and \ref{no-unitary} by applying Theorem \ref{sharp-realroot}.

\begin{proof}
[Proof of Theorem \ref{no-unitary}]Theorem \ref{no-unitary} follows
from Theorems \ref{thm4} and \ref{sharp-realroot}.
\end{proof}

\begin{proof}
[Proof of Theorem \ref{no-even-solution}]Theorem \ref{no-even-solution} follows from Theorems
\ref{thm6} and \ref{sharp-realroot}.
\end{proof}

Together with
Eremenko and Gabrielov's result Theorem A and our Theorem \ref{thm6}, we
immediately obtain

\begin{theorem}
\label{thm7} Let $n_{k}\in \mathbb{Z}_{\geq 0}$ for all $k$ with $\max_k n_{k}\geq1$. Suppose either
\begin{equation}
\label{c1-1}\frac{n_{1}+n_{2}-n_{0}-n_{3}}{2}\geq1,\quad n_{1}\geq1,\quad
n_{2}\geq1
\end{equation}
or
\begin{equation}
\label{c2-1}\frac{n_{1}+n_{2}-n_{0}-n_{3}}{2}\leq-1,\quad n_{0}\geq1,\quad
n_{3}\geq1.
\end{equation}
Then there exists $\tau \in i\mathbb{R}_{>0}$ such that $Q^{(n_{0},n_{1}%
,n_{2},n_{3})}(\cdot;\tau)$ has either multiple zeros or complex zeros.
\end{theorem}

Theorem \ref{thm7} indicates that our condition on $n_k$ in Theorem \ref{sharp-realroot} is sharp, namely Theorem \ref{sharp-realroot11} holds.

\begin{corollary}[=Theorem \ref{sharp-realroot11}]\label{sharp-realroot1}
Let $n_{k}\in \mathbb{Z}_{\geq 0}$ for all $k$ with $\max_k n_{k}\geq1$. Then all the zeros of $Q^{(n_{0},n_{1},n_{2},n_{3})}(\cdot;\tau)$ are real
and distinct for each $\tau \in i\mathbb{R}_{>0}$ if and only if $(n_{0}%
,n_{1},n_{2}$, $n_{3})$ satisfies neither (\ref{c1-1}) nor (\ref{c2-1}).
\end{corollary}

In general, it is very difficult to prove such an optimal algebraic result for the spectral polynomial $Q^{(n_{0},n_{1},n_{2},n_{3})}(\cdot;\tau)$ of the Treibich-Verdier potential (\ref{Tre}). Corollary \ref{sharp-realroot1} is a beautiful application of Eremenko and Gabrielov's result (Theorem A, via geometric approach) and our result (Theorem \ref{thm6}, via analytic approach).

\section{Real distinct roots of $Q^{(n_{0},n_{1},n_{2},n_{3})}$}

\label{real-distinct}

The purpose of this section is to prove Theorem \ref{sharp-realroot}. As pointed out by Corollary \ref{sharp-realroot1}, our condition on $n_k$ in Theorem \ref{sharp-realroot} is optimal. A non-optimal version of Theorem \ref{sharp-realroot} was proved in \cite[Theorem 1.1]{CKLT}.

\subsection {}
As in \cite{CKLT}, first we need to investigate polynomial solutions of
\begin{equation}
\frac{d^{2}y}{dx^{2}}+\left(  \frac{\gamma_{1}}{x-t_{1}}+\frac{\gamma_{2}%
}{x-t_{2}}+\frac{\gamma_{3}}{x-t_{3}}\right)  \frac{dy}{dx}+\frac{\alpha
\beta(x-t_{3})-q}{\prod_{j=1}^{3}(x-t_{j})}y=0, \label{eq:Heunt1t2t3}%
\end{equation}
where
\begin{equation}\label{t-i}t_{1}\neq t_{2}\neq t_{3}\neq t_{1}, \quad \gamma_{3}\not \in -{\mathbb{Z}
}_{\geq0},\end{equation}
\begin{equation}\label{alpha-gamma} \alpha=-N \;\text{with}\; N\in\mathbb{Z}_{\geq 0},\quad\alpha+\beta
+1=\gamma_{1}+\gamma_{2}+\gamma_{3}.\end{equation}
It is a Fuchsian equation on $\mathbb{CP}^{1}$ with four regular singularities
$\{t_{1},t_{2},t_{3},\infty \}$, with the exponents being $0, 1-\gamma_{j}$ at $t_j$ and $\alpha,\beta$ at $\infty$. Set
\begin{equation}
y=\sum_{m=0}^{\infty}c_{m}(x-t_{3})^{m},\quad \text{where}\;c_{0}=1,
\end{equation}
and substitute it to the differential equation which is multiplied by
$\prod_{j=1}^{3}(x-t_{j})$ to (\ref{eq:Heunt1t2t3}). Then the
coefficients satisfy the following recursive relations:
\begin{equation}
(t_{1}-t_{3})(t_{2}-t_{3})\gamma_{3}c_{1} =qc_{0}=q,\label{eq:Hlci}\end{equation}
\begin{align*}
(t_{1}-t_{3})(t_{2}-t_{3})(m+1)(m+\gamma_{3})c_{m+1}  &  =-(m-1+\alpha)(m-1+\beta)c_{m-1}\nonumber \\
+[m\{(m-1+\gamma_{3})(t_{1}+t_{2}-2t_{3})+  &  (t_{2}-t_{3})\gamma_{1}%
+(t_{1}-t_{3})\gamma_{2}\}+q]c_{m}.
\end{align*}
Consequently, it is easy to see that \emph{$c_{r}$ is a polynomial in $q$ of
degree $r$ and we denote it by $c_{r}(q)$}.

Let $q_{0}$ be a solution to the equation $c_{N+1} (q)=0$, where $N$ is given by (\ref{alpha-gamma}). Then
it follows from (\ref{eq:Hlci}) for $m=N+1$ that $c_{N+2} (q_{0})=0$. By
applying (\ref{eq:Hlci}) for $m=N+2, N+3, \dots$, we have $c_{m} (q_{0})=0$
for $m\geq N+3$. Hence, if $c_{N+1} (q_{0})=0$, then
(\ref{eq:Heunt1t2t3}) have a non-zero polynomial solution. More precisely, we
obtain the following proposition.

\begin{proposition}
\label{prop:Heunpolym} Suppose (\ref{t-i})-(\ref{alpha-gamma}) hold.
If $q$ is a solution to the equation $c_{N+1} (q)=0$, then the differential
equation (\ref{eq:Heunt1t2t3}) have a non-zero polynomial solution of degree
no more than $N$.
\end{proposition}

Now we restrict to the case that all the parameters are \emph{real}. Then $c_{r}(q)$ is a polynomial of $q$ with real coefficients.

\begin{theorem}\cite{CKLT}
\label{thm:Sturm} Let $t_j, \gamma_j$ are all real. Assume that $ \alpha=-N$ with $N \in{\mathbb{Z}}_{\geq0} $, $\beta =\gamma_{1}+
\gamma_{2} +\gamma_{3} +N-1>0$, $\gamma_{3} >0$ and $(t_{1} -t_{3})(t_{2}- t_{3}) <0$. Then the
equation $c_{N+1} (q)=0$ has all its roots real and unequal.
\end{theorem}

The above theorem was proved in \cite{CKLT} by applying the standard method of Sturm sequence. In this paper, we prove the following result.

\begin{theorem}
\label{thm:Sturm1} Let $t_j, \gamma_j$ are all real. Assume that there are integers $n_0\ge n_3\geq 1$ such that
\begin{equation}\label{t-0}(t_{1} -t_{3})(t_{2}- t_{3}) <0,\quad -\alpha=N=n_0+n_3,\end{equation}
\begin{equation}\label{tt-0}\gamma_3=\beta=\gamma_{1}+\gamma_{2}+\gamma_{3}-1-\alpha=\tfrac{1}{2}-n_3<0.\end{equation}
Then the
equation $c_{N+1} (q)=0$ has all its roots real and unequal.
\end{theorem}

\begin{proof}Under our assumptions (\ref{t-0})-(\ref{tt-0}), we have
\[-(m-1+\alpha)>0,\quad \forall\, m\in [1, N],\]
\[m+\gamma_3=m+\beta=m+\tfrac{1}{2}-n_3\begin{cases}<0\quad\text{if}\; m\leq n_3-1,\\>0\quad\text{if}\; m\geq n_3.\end{cases}\]
Together with the recursive formula (\ref{eq:Hlci}), we easily obtain the following properties:

{\bf (P1)} Up to a positive constant, the leading term in $c_{m}(q)$ is
\[\begin{cases}q^m\quad\text{if}\; 1\leq m\leq n_3,\\(-1)^{m-n_3}q^m\quad\text{if}\;  n_3\leq m\leq N+1.\end{cases}\]

{\bf (P2)} If $c_{m}(q)=0$ and $c_{m-1} (q)\neq 0$ for $q\in \mathbb{R}$, then
\[c_{m+1} (q)c_{m-1} (q)\begin{cases}<0\quad\text{if}\; 1\leq m\leq N, m\neq n_3,\\>0\quad\text{if}\;  m= n_3.\end{cases}\]
Therefore, $c_m(q)$ is not a Sturm sequence. However, we can still show that the polynomial $c_{m}(q)$ $(1\leq m\leq N+1)$ has $m$ real
distinct roots $s_{i}^{(m)}$ $(i=1,\dots, m)$ such that
\[
s_{1}^{(m)}<s_{1}^{(m-1)}<s_{2}^{(m)}<s_{2}^{(m-1)}<\dots<s_{m-1}%
^{(m)}<s_{m-1}^{(m-1)}<s_{m}^{(m)}%
\]
by induction on $m$. The case $m=1$ is trivial. Let $1\le k \le N$ and
assume that the statement is true for $m\leq k$. From the assumption of the
induction,
\begin{equation}\label{s-k}
s_{1}^{(k)}<s_{1}^{(k-1)}<s_{2}^{(k)}<s_{2}^{(k-1)}<\dots
<s_{k-1}^{(k)}<s_{k-1}%
^{(k-1)}<s_{k}^{(k)}.
\end{equation}

{\bf Case 1}. We consider $k\leq n_3-1$.

Then {\bf (P1)} implies
\begin{equation}\label{ck-1}
\lim_{q \rightarrow-\infty} c_{k-1}(q) = (-1)^{k-1}\infty, \; \lim_{q \rightarrow
+\infty} c_{k-1}(q) = + \infty.
\end{equation}
Since $s_{j}^{(k-1)}$, $1\le j\le k-1$, are all the roots of  $c_{k-1}$, it follows from (\ref{s-k}) and (\ref{ck-1}) that
\begin{equation}\label{c-k-1-k}
c_{k-1}(s_i^{(k)})\sim (-1)^{k-i},\quad \forall i\in [1,k].
\end{equation}
Here $c\sim (-1)^j$ means $c=(-1)^j\tilde{c}$ for some $\tilde{c}>0$. Then we see from {\bf (P2)} that
\[c_{k+1}(s_i^{(k)})\sim (-1)^{k+1-i},\quad \forall i\in [1,k].\]
On the other hand, {\bf (P1)} implies
\[\lim_{q \rightarrow-\infty} c_{k+1}(q) = (-1)^{k+1}\infty, \; \lim_{q \rightarrow
+\infty} c_{k+1}(q) = + \infty.\]
From here,
it follows from the intermediate value theorem that the polynomial
$c_{k+1}(q)$ has $k+1$ real distinct roots $s_{i}^{(k+1)} $ $(1\leq i\leq
k+1)$ such that \begin{equation}\label{sk1}s_{1}^{(k+1)}<s_{1}^{(k)}<s_{2}^{(k+1)}%
<s_{2}^{(k)}<\dots<s_{k}^{(k+1)}<s_{k}^{(k)}<s_{k+1}^{(k+1)}.\end{equation}

{\bf Case 2}. We consider $k=n_3$.

Then (\ref{c-k-1-k}) still holds, and so {\bf (P2)} gives
\[c_{k+1}(s_i^{(k)})\sim (-1)^{k-i},\quad \forall i\in [1,k],\]
which is different from Case 1! However, {\bf (P1)} says that the leading term of $c_{k+1}=c_{n_3+1}$ is $-q^{k+1}$ (up to a positive constant), which implies
\[\lim_{q \rightarrow-\infty} c_{k+1}(q) = (-1)^{k}\infty, \; \lim_{q \rightarrow
+\infty} c_{k+1}(q) = - \infty.\]
This is also different from Case 1! Thanks to these two facts, we see again that $c_{k+1}(q)$ has $k+1$ real distinct roots $s_{i}^{(k+1)} $ $(1\leq i\leq
k+1)$ such that (\ref{sk1}) holds.

{\bf Case 3}. We consider $n_3+1\leq k\leq N$.

Then {\bf (P1)} says that the leading term of $c_{k-1}$ is $(-1)^{k-1-n_3}q^{k-1}$, which implies
\[\lim_{q \rightarrow-\infty} c_{k-1}(q) = (-1)^{-n_3}\infty, \; \lim_{q \rightarrow
+\infty} c_{k-1}(q) = (-1)^{k-1-n_3} \infty.\]
From here and (\ref{s-k}), we obtain
\[c_{k-1}(s_i^{(k)})\sim (-1)^{i-1-n_3},\quad \forall i\in [1,k],\]
and so {\bf (P2)} implies
\[c_{k+1}(s_i^{(k)})\sim (-1)^{i-n_3},\quad \forall i\in [1,k].\]
Recall {\bf (P1)} that the leading term of $c_{k+1}$ is $(-1)^{k+1-n_3}q^{k+1}$, which gives
\[\lim_{q \rightarrow-\infty} c_{k+1}(q) = (-1)^{-n_3}\infty, \; \lim_{q \rightarrow
+\infty} c_{k+1}(q) = (-1)^{k+1-n_3} \infty.\]
Therefore, we conclude again that $c_{k+1}(q)$ has $k+1$ real distinct roots $s_{i}^{(k+1)} $ $(1\leq i\leq
k+1)$ such that (\ref{sk1}) holds.

This proves that $c_{N+1} (q)=0$ has all its roots real and unequal.
The proof is complete.\end{proof}

\subsection {}

Recalling GLE (\ref{GLE-1}),
we let $y(z)$ be a solution of GLE,
\begin{equation}
\bigg(\frac{d^{2}}{dz^{2}}-\sum_{k=0}^{3}n_{k}(n_{k}%
+1)\wp(z+\tfrac{\omega_{k}}{2};\tau)-E\bigg)  y(z)=0. \label{InoEF}%
\end{equation}
Set $x=\wp(z)$ and recall $e_k$ defined in Remark \ref{remark4-2}. Applying the formula
\[
\wp(z+\tfrac{\omega_{i}}{2})=e_{i}+\frac{(e_{i}-e_{i^{\prime}})(e_{i}%
-e_{i^{\prime \prime}})}{\wp(z)-e_{i}},\; \, \text{where $\{i,i^{\prime
},i^{\prime \prime}\}=\{1,2,3\}$,}
\]
it is easy to see that equation (\ref{InoEF}) is equivalent to
\begin{align}
&  \left \{  \frac{d^{2}}{dx^{2}}+\frac{1}{2}\bigg(  \sum_{i=1}^3\frac{1}{x-e_{i}}\bigg)  \frac{d}{dx}-\frac{1}{4\prod_{j=1}%
^{3}(x-e_{j})}\bigg(  \tilde{C}+\right. \label{Heun}\\
&  \; \left.n_{0}(n_{0}+1)x+\sum_{i=1}^{3}n_{i}%
(n_{i}+1)\frac{(e_{i}-e_{i^{\prime}})(e_{i}-e_{i^{\prime \prime}})}{x-e_{i}%
}\bigg)  \right \}  \tilde{f}(x)=0,\nonumber
\end{align}
where $\tilde{f}(\wp(z))=y(z)$ and $\tilde{C}=E+\sum_{i=1}^{3}n_{i}%
(n_{i}+1)e_{i}$. Note that $e_{1}+e_{2}+e_{3}=0$. It is easy to see that the
Riemann scheme of equation (\ref{Heun}) is
\[%
\begin{Bmatrix}
e_{1} & e_{2} & e_{3} & \infty \\
\frac{-n_{1}}{2} & \frac{-n_{2}}{2} & \frac{-n_{3}}{2} & \frac{-n_{0}}{2}\\
\frac{n_{1}+1}{2} & \frac{n_{2}+1}{2} & \frac{n_{3}+1}{2} & \frac{n_{0}+1}{2}%
\end{Bmatrix}
.
\]

Let $\tilde{\alpha}_{i}\in\{-n_{i}/2,(n_{i}+1)/2\}$ for each $i\in \{0,1,2,3\}$ such that $N:=-\sum \tilde{\alpha}_{i}\in\mathbb{Z}_{\geq 0}$.
Set
\[
\Phi^{(\tilde{\alpha}_{1},\tilde{\alpha}_{2},\tilde{\alpha}_{3})}%
(x)=\prod_{j=1}^{3}(x-e_{j})^{\tilde{\alpha}_{j}}\quad \text{and}\quad \tilde
{f}(x)=\Phi^{(\tilde{\alpha}_{1},\tilde{\alpha}_{2},\tilde{\alpha}_{3}%
)}(x)f(x).
\]
Then $\tilde{f}(x)$ solves equation (\ref{Heun}) is equivalent to that $f(x)$
satisfies
\begin{align}
\frac{d^{2}f(x)}{dx^{2}}&+  \sum_{i=1}^{3}\frac{2\tilde{\alpha}_{i}+\frac
{1}{2}}{x-e_{i}}\frac{df(x)}{dx}+\left(  \frac{(\sum_{i=1}^3\tilde{\alpha}_{i}
-\frac{n_{0}}{2})(\sum_{i=1}^3\tilde{\alpha}_{i}+\frac{n_{0}+1}{2})x}{(x-e_{1}%
)(x-e_{2})(x-e_{3})}\right. \nonumber \\
&  \left.  -\frac{\frac{E}{4}+e_{1}(\tilde{\alpha}_{2}+\tilde{\alpha}_{3}%
)^{2}+e_{2}(\tilde{\alpha}_{1}+\tilde{\alpha}_{3})^{2}+e_{3}(\tilde{\alpha
}_{1}+\tilde{\alpha}_{2})^{2}}{(x-e_{1})(x-e_{2})(x-e_{3})}\right)  f(x)=0.
\label{Heun2}%
\end{align}
This equation is in the form of equation (\ref{eq:Heunt1t2t3}) by setting
\[
\alpha=\tilde{\alpha}_{0}+\tilde{\alpha}_{1}+\tilde{\alpha}_{2}+\tilde{\alpha
}_{3}=-N,\quad N\in\mathbb{Z}_{\geq 0},
\]%
\[
\beta=-\tilde{\alpha}_{0}+\tfrac{1}{2}+\tilde{\alpha}_{1}+\tilde{\alpha}%
_{2}+\tilde{\alpha}_{3},
\]%
\[
\gamma_{i}=2\tilde{\alpha}_{i}+\tfrac{1}{2},\quad t_{i}=e_{i},\quad i=1,2,3,
\]%
\[
q= \tfrac{E}{4}+e_{1}(\tilde{\alpha}_{2}+\tilde{\alpha}_{3})^{2}%
+e_{2}(\tilde{\alpha}_{1}+\tilde{\alpha}_{3})^{2}+e_{3}(\tilde{\alpha}%
_{1}+\tilde{\alpha}_{2})^{2}-e_{3}\alpha \beta .
\]
It is
well known that $e_{j}=e_{j}(\tau)\in \mathbb{R}$ and $e_{1}>e_{3}>e_{2}$
for $\tau \in i\mathbb{R}_{>0}$, i.e.
\begin{equation}
(e_{1}-e_{3})(e_{2}-e_{3})<0\quad \text{if}\; \tau \in i\mathbb{R}_{>0}.
\label{eqe}%
\end{equation}
Write $f(x)=\sum_{r=0}^{\infty}c_{r}(x-e_{3})^{r}$ with $c_{0}=1$, then
$c_{r}$ is a polynomial in $q$ and equivalently in $E$ of degree $r$. We denote it by $c_{r}(E)$. Then
it follows from Proposition \ref{prop:Heunpolym}  that if $c_{N+1}(E)=0$, then
the differential equation (\ref{Heun}) has a "polynomial" solution $\tilde
{f}(x)=\Phi^{(\tilde{\alpha}_{1},\tilde{\alpha}_{2},\tilde{\alpha}_{3}%
)}(x)f(x)$ in the sense that $f(x)$ is a polynomial of degree no more than $N$.

Let $P_{\tilde{\alpha}_{0}, \tilde{\alpha}_{1}, \tilde{\alpha}_{2} ,
\tilde{\alpha}_{3}}(E)$ be \emph{the monic polynomial obtained by normalising
$c_{N+1} (E) $}. Then
\[
\deg P_{\tilde{\alpha}_{0}, \tilde{\alpha}_{1}, \tilde{\alpha}_{2} ,
\tilde{\alpha}_{3}}(E)= N+1= -\tilde{\alpha}_{0}-\tilde{\alpha}_{1}%
-\tilde{\alpha}_{2}-\tilde{\alpha}_{3}+1.
\]

Recall $n_{k}\in{\mathbb{Z}}_{\geq0}$ for all $k$.
We recall the following important result from \cite{Tak1}, which establishes
the precise relation between the spectral polynomial $Q^{(n_{0},n_{1}%
,n_{2},n_{3})}(E)$ and the aforementioned polynomial $P_{\tilde{\alpha}%
_{0},\tilde{\alpha}_{1},\tilde{\alpha}_{2},\tilde{\alpha}_{3}}(E)$.
For our purpose, we only consider the case that $\sum_k n_{k}$ is even, then $Q(E)
=Q^{(n_{0},n_{1},n_{2},n_{3})}(E)$ is
written as $Q(E)=P^{(0)} (E) P^{(1)} (E) P^{(2)} (E) P^{(3)} (E)$,
where{\allowdisplaybreaks
\begin{align*}
&  P^{(0)} (E) = P_{-n_{0}/2, -n_{1}/2, -n_{2}/2, -n_{3}/2}(E),\\
&  P^{(1)} (E) = \left \{
\begin{array}
[c]{ll}%
P_{-n_{0}/2, -n_{1}/2, (n_{2} +1)/2, (n_{3} +1)/2}(E), & n_{0} +n_{1} \geq
n_{2} +n_{3} +2 ,\\
1 , & n_{0} +n_{1} = n_{2} +n_{3} ,\\
P_{(n_{0}+1)/2, (n_{1} +1)/2, -n_{2} /2, -n_{3} /2}(E), & n_{0} +n_{1} \leq
n_{2} +n_{3} -2,
\end{array}
\right. \\
&  P^{(2)} (E) = \left \{
\begin{array}
[c]{ll}%
P_{-n_{0}/2, (n_{1}+1)/2, -n_{2} /2, (n_{3} +1)/2}(E), & n_{0} +n_{2} \geq
n_{1} +n_{3} +2 ,\\
1 , & n_{0} +n_{2} = n_{1} +n_{3} ,\\
P_{(n_{0}+1)/2, -n_{1}/2, (n_{2} +1)/2, -n_{3} /2}(E), & n_{0} +n_{2} \leq
n_{1} +n_{3} -2,
\end{array}
\right. \\
&  P^{(3)} (E) = \left \{
\begin{array}
[c]{ll}%
P_{-n_{0}/2, (n_{1}+1)/2, (n_{2} +1)/2, -n_{3} /2}(E), & n_{0} +n_{3} \geq
n_{1} +n_{2} +2 ,\\
1 , & n_{0} +n_{3} = n_{1} +n_{2} ,\\
P_{(n_{0}+1)/2, -n_{1} /2, -n_{2} /2, (n_{3} +1)/2}(E), & n_{0} +n_{3} \leq
n_{1} +n_{2} -2.
\end{array}
\right.
\end{align*}
}Furthermore, it was shown in \cite[Theorem 3.2]{Tak1} that the equations
$P^{(i)}(E)=0$ and $P^{(j)}(E)=0$ $(i\neq j)$ \textit{do not have common
solutions}.

Recall the following result proved in \cite{CKLT}.

\begin{theorem}\cite{CKLT}
\label{thm:Qrealdist} Suppose $n_{0}, n_{1}, n_{2} , n_{3} \in{\mathbb{Z}%
}_{\geq0}$ with $\max_k n_k\geq 1$. If $n_{3}=0$, $n_{0} \geq n_{1} +n_{2} -1$ and $\tau \in i
{\mathbb{R}}_{>0}$, then the zeros of $Q^{(n_{0}%
,n_{1},n_{2},n_{3})} (E)$ are all real and unequal.
\end{theorem}

Here we prove the following analogous result for new cases.

\begin{theorem}
\label{thm:Qrealdist1} Suppose $n_{0}, n_{1}, n_{2} , n_{3} \in{\mathbb{Z}%
}_{\geq0}$ satisfying $\max_k n_k\geq 1$ and \begin{equation}\label{nkequal}n_{0}+n_3 = n_{1} +n_{2}.\end{equation} Then for $\tau \in i
{\mathbb{R}}_{>0}$, the zeros of $Q^{(n_{0}%
,n_{1},n_{2},n_{3})} (E)$ are all real and unequal.
\end{theorem}

\begin{proof}
By changing variable $z\mapsto z+\frac{\omega_k}{2}$ in GLE (\ref{GLE-1}), we have
\begin{align}\label{Qnn}
&Q^{(n_{0},n_{1},n_{2},n_{3})}(E)=Q^{(n_{1},n_{0},n_{3},n_{2}
)}(E)\\
=&Q^{(n_{2},n_{3},n_{0},n_{1})}(E)=Q^{(n_{3},n_{2},n_{1},n_{0}
)}(E).\nonumber\end{align}
Therefore, we may always assume $n_0=\max n_k$ and then (\ref{nkequal}) implies $n_3=\min n_k$. If $n_3=0$, then this theorem follows from Theorem \ref{thm:Qrealdist}. Therefore, we only consider $n_3\ge 1$, i.e.
\begin{equation}\label{maxmin}n_0=\max n_k,\quad n_3=\min n_k\geq 1.\end{equation}

Note that $\sum n_k$ is even.
We only need to show that the zeros of each polynomial $P^{(j)}(E)$,
$j\in \{0,1,2,3\}$, are all real and unequal.

Since $P^{(0)}(E)=P_{-n_{0}/2, -n_{1}/2, -n_{2}/2, -n_{3}/2}(E)$, we have
\[\alpha=-\tfrac{n_{0}}{2}-\tfrac{n_{1}}{2}-\tfrac{n_{2}}{2}-\tfrac{n_{3}}{2}=-n_0-n_3,\quad N=-\alpha=n_0+n_3,\]
\[\gamma_3=\tfrac{1}{2}-n_3<0,\]
\begin{align*}\beta &=-\tilde{\alpha}_{0}+\tfrac{1}{2}+\tilde{\alpha}_{1}+\tilde{\alpha}%
_{2}+\tilde{\alpha}_{3}=\tfrac{n_{0}}{2}+\tfrac{1}{2}-\tfrac{n_{1}}{2}-\tfrac{n_{2}}{2}-\tfrac{n_{3}}{2}\\
&=\tfrac{1}{2}-n_3=\gamma_3<0.\end{align*}
Then we can apply Theorem \ref{thm:Sturm1} to see that the zeros of $P^{(0)}(E)$ are all real and distinct.

Clearly (\ref{nkequal}) and (\ref{maxmin}) imply $n_0+n_1\geq n_2+n_3$. Since $P^{(1)} (E)=1$ for $n_{0} +n_{1} = n_{2} +n_{3}$, we only need to consider $n_0+n_1\geq n_2+n_3+2$ and so
\[P^{(1)} (E)
=P_{-n_{0}/2, -n_{1}/2, (n_{2} +1)/2, (n_{3} +1)/2}(E).
\]
Then
\[\alpha=-\tfrac{n_{0}}{2}-\tfrac{n_{1}}{2}+\tfrac{n_{2}+1}{2}+\tfrac{n_{3}+1}{2}=n_2+1-n_0\leq 0,\quad N=-\alpha\geq 0,\]
\[\gamma_3=\tfrac{3}{2}+n_3>0,\]
\begin{align*}\beta &=-\tilde{\alpha}_{0}+\tfrac{1}{2}+\tilde{\alpha}_{1}+\tilde{\alpha}%
_{2}+\tilde{\alpha}_{3}=\tfrac{n_{0}}{2}+\tfrac{1}{2}-\tfrac{n_{1}}{2}+\tfrac{n_{2}+1}{2}+\tfrac{n_{3}+1}{2}\\
&=\tfrac{3}{2}+n_2>0.\end{align*}
Thus by Theorem \ref{thm:Sturm}, we see that the zeros of $P^{(1)}(E)$ are all real and distinct. The same argument shows that the zeros of $P^{(2)}(E)$ are all real and distinct. Finally, we note that $P^{(3)}(E)=1$.

In conclusion, the zeros of $Q^{(n_{0}%
,n_{1},n_{2},n_{3})} (E)$ are all real and distinct.
\end{proof}

\begin{theorem}
\label{thm:Qrealdist2} Suppose $n_{0}, n_{1}, n_{2} , n_{3} \in{\mathbb{Z}%
}_{\geq0}$ satisfying \begin{equation}\label{nkequal1}n_{0}+n_3 = n_{1} +n_{2}\pm 1.\end{equation} Then for $\tau \in i
{\mathbb{R}}_{>0}$, the zeros of $Q^{(n_{0}%
,n_{1},n_{2},n_{3})} (E)$ are all real and unequal.
\end{theorem}

\begin{proof} By (\ref{Qnn}) we only need to consider the case
\begin{equation}\label{equu}n_{0}+n_3 = n_{1} +n_{2}+ 1\;\;\text{and}\;\; n_0\geq n_3.\end{equation}
Again by Theorem \ref{thm:Qrealdist}, we only need to consider $n_3\ge 1$.
Since $\sum n_k$ is odd, we define
{\allowdisplaybreaks
\begin{align*}
& l_{0}= (n_{0}+n_{1}+n_{2}+n_{3}+1)/2=n_0+n_3>0,\\
& l_{1}= (n_{0}+n_{1}-n_{2}-n_{3}-1)/2=n_1-n_3,\\
& l_{2}= (n_{0}-n_{1}+n_{2}-n_{3}-1)/2=n_2-n_3,\\
& l_{3}= (n_{0}-n_{1}-n_{2}+n_{3}-1)/2=0.
\end{align*}
}%
Then it was proved in \cite[Section 4]{Tak5} (see also \cite[Section 3]{CKLT}) that
\[Q^{(n_0,n_1,n_2,n_3)}(E)=Q^{(l_0,l_1,l_2,l_3)}(E).\]
Note that if $l_{1}<0$ and $l_2<0$, then $n_1\le n_3-1$ and $n_2\le n_3-1$, which contradict with our assumption (\ref{equu}). Thus there are three cases.

{\bf Case 1}. $l_1\ge 0$ and $l_2\ge 0$.

Then
\[l_{0}-l_{1}-l_2+1=2n_3+2>0.\]Thus, we can apply Theorem \ref{thm:Qrealdist} to $Q^{(l_0,l_1,l_2,l_3)}(E)$ and obtain that the zeros of $Q^{(n_{0}%
,n_{1},n_{2},n_{3})} (E)$ are all real and unequal.

{\bf Case 2}. $l_1\ge 0$ and $l_2< 0$.

Then $-1-l_2\ge 0$. Since \[\frac{d^{2}}{dz^{2}}-\sum_{k=0}^{3}%
l_{k}(l_{k}+1)\wp(z+\tfrac{\omega_{k}}{2})\]
 is \emph{invariant} by replacing $l_{2}$ to
$-l_{2}-1$, we obtain
\[Q^{(n_0,n_1,n_2,n_3)}(E)=Q^{(l_0,l_1,l_2,l_3)}(E)=Q^{(l_0,l_1,-l_2-1,l_3)}(E).\]
Since
\[l_{0}-l_{1}-(-l_2-1)+1=2n_2+3>0,\]
we can apply Theorem \ref{thm:Qrealdist} to $Q^{(l_0,l_1,-l_2-1,l_3)}(E)$ and obtain that the zeros of $Q^{(n_{0}%
,n_{1},n_{2},n_{3})} (E)$ are all real and unequal.

{\bf Case 3}. $l_1<0$ and $l_2\ge 0$. The proof is the same as Case 2.

In conclusion, the zeros of $Q^{(n_{0}%
,n_{1},n_{2},n_{3})} (E)$ are all real and unequal. The proof is complete.
\end{proof}

We are in the position to prove Theorem \ref{sharp-realroot}.

\begin{proof}[Proof of Theorem \ref{sharp-realroot}]
Since neither (\ref{c1-11}) nor (\ref{c2-11}) hold, we have one of the followings hold:
\begin{equation}\label{ffff}
n_{1}+n_{2}-n_{0}-n_{3}\in \{0,1,-1\},
\end{equation}
\begin{equation}\label{ffff1}
n_{1}+n_{2}-n_{0}-n_{3}\geq 2, \quad\text{either $n_{1}=0$ or $n_{2}=0$},
\end{equation}
\begin{equation}\label{ffff2}
n_{0}+n_{3}-n_{1}-n_{2}\geq 2, \quad\text{either $n_{0}=0$ or $n_{3}=0$}.
\end{equation}
If (\ref{ffff}) holds, the assertion follow from Theorems \ref{thm:Qrealdist1} and \ref{thm:Qrealdist2}.
If (\ref{ffff2}) holds, by (\ref{Qnn}) we may assume $n_3=0$ and so the assertion follows from Theorem \ref{thm:Qrealdist}. Finally, we see from  (\ref{Qnn}) that the case (\ref{ffff1}) is equivalent to the case (\ref{ffff2}). The proof is complete.
\end{proof}

\section{Application to pre-modular form}

\label{premodularform}

In this section, we apply Theorem \ref{no-sol-n} to the pre-modular form $Z_{r,s}^{(n)}(\tau)$ introduced by \cite{CLW2}.
As pointed out in the introduction, the solvability of the curvature equation (\ref{eq1-1}), i.e.
\begin{equation}\label{mfe1}\Delta u+e^u=8n\pi\delta_0\quad\text{on}\; E_{\tau},\end{equation}
depends essentially on the moduli $\tau$ of the flat torus $E_{\tau}$ and is intricate from the PDE point of view. To settle this challenging problem, Chai, Wang and the second author studied it from the viewpoint of algebraic geometry. They
developed a theory to connect this PDE problem with the Lam\'{e} equation (i.e. GLE (\ref{GLE-1}) with $n_0=n$ and $n_k=0$ for $k\in \{1,2,3\}$)
\[
y"(z)=[n(n+1)\wp(z;\tau)+E]y(z)
\]
and pre-modular forms. In particular, Wang and the second author \cite{CLW2} proved the following important result.

\begin{theorem}\cite{CLW2}\label{LW-pre} There exists a pre-modular form $Z_{r,s}^{(n)}(\cdot)$ of weight $\frac{n(n+1)}{2}$ such that (\ref{mfe1}) on $E_{\tau}$ has solutions if and only if $Z_{r,s}^{(n)}(\tau)=0$ for some $(r,s)\in \mathbb{R}^{2}\setminus\frac{1}{2}\mathbb{Z}^{2}$.
\end{theorem}

The pre-modular form $Z_{r,s}^{(n)}(\tau)$ is holomorphic in $\tau$ for each $(r,s)\in\mathbb{R}^{2}\setminus\frac{1}{2}\mathbb{Z}^2$, and possess the following properties (see \cite{CLW2}):

\begin{itemize}
\item[(i)] $Z_{r,s}^{(n)}(\tau)=\pm
Z_{m\pm r,n\pm s}^{(n)}(\tau)$ for any $(m,n)\in\mathbb{Z}^{2}$.

\item[(ii)] For any $\gamma=
\begin{pmatrix}
a & b\\
c & d
\end{pmatrix}
\in SL(2,\mathbb{Z})$, we define $\tau^{\prime}=\gamma \cdot \tau:=\frac{a\tau+b}{c\tau+d}$ and $(s^{\prime}%
,r^{\prime}):=(s,r)\cdot \gamma^{-1}$. Then \[Z_{r^{\prime},s^{\prime}}^{(n)}(\tau^{\prime})
    =(c\tau+d)^{\frac{n(n+1)}{2}}Z_{r,s}%
^{(n)}(\tau).\]
\end{itemize}
\noindent In particular, when
$(r,s)\in Q_{N}$ is a $N$-torsion point for some $N\in\mathbb{N}_{\geq 3}$, where
\begin{equation}
Q_{N}:= \left \{  \left.  \left(  \tfrac{k_{1}}{N},\tfrac{k_{2}}%
{N}\right)  \right \vert \gcd(k_{1},k_{2},N)=1,\text{ }0\leq k_{1},k_{2}\leq
N-1\right \}  , \label{q-n}%
\end{equation}
and $\gamma\in \Gamma(N):=\{\gamma\in SL(2,\mathbb{Z})|\gamma\equiv I_2\operatorname{mod}N\}$, then $(r',s')\equiv(r,s)$ mod $\mathbb{Z}^2$ and so
\[Z_{r,s}^{(n)}\left(\tfrac{a\tau+b}{c\tau+d}\right)
=(c\tau+d)^{\frac{n(n+1)}{2}}Z_{r,s}^{(n)}(\tau).\]
Thus $Z_{r,s}^{(n)}(\tau)$ is \emph{a modular form} of weight $\frac{n(n+1)}{2}$ with respect to the principal congruence subgroup $\Gamma(N)$. Due to this reason, $Z_{r,s}^{(n)}(\tau)$ are called \emph{pre-modular forms} in this paper as in \cite{CLW2}.

For $n\le 4$, the explicit expression of $Z_{r,s}^{(n)}(\tau)$ is known; see \cite{CLW2}.
Let $\zeta(z;\tau):=-\int^{z}\wp(\xi;\tau)d\xi$ be the Weierstrass zeta
function, which is odd and has two quasi-periods $\eta_{k}(\tau):=2\zeta(\frac{\omega_k}{2};\tau)$, $k=1,2$:
\[
\eta_{1}(\tau)=\zeta(z+1;\tau)-\zeta(z;\tau),\text{ \ }\eta_{2}(\tau
)=\zeta(z+\tau;\tau)-\zeta(z;\tau). \]
Define
\[
Z=Z_{r,s}(\tau)  :=\zeta(r+s\tau;\tau)-r\eta_{1}(\tau)-s\eta_{2}%
(\tau).
\]
Then it is known \cite{CLW2} that (write $\wp=\wp(r+s\tau;\tau)$ and $\wp^{\prime
}=\wp^{\prime}(r+s\tau;\tau)$ for convenience):
$Z_{r,s}^{(1)}(\tau)=Z_{r,s}(\tau)$,
\[
Z_{r,s}^{(2)}(\tau)=Z^{3}-3\wp Z-\wp^{\prime},
\]%
\begin{align*}
Z_{r,s}^{(3)}(\tau)=  &  Z^{6}-15\wp Z^{4}-20\wp^{\prime}Z^{3}+\left(
\tfrac{27}{4}g_{2}-45\wp^{2}\right)  Z^{2}\\
&  -12\wp \wp^{\prime}Z-\tfrac{5}{4}(\wp^{\prime})^{2}.
\end{align*}
{\allowdisplaybreaks%
\begin{align*}
Z_{r,s}^{(4)}(\tau)=  &  Z^{10}-45\wp Z^{8}-120\wp^{\prime}Z^{7}+(\tfrac
{399}{4}g_{2}-630\wp^{2})Z^{6}-504\wp \wp^{\prime}Z^{5}\\
&  -\tfrac{15}{4}(280\wp^{3}-49g_{2}\wp-115g_{3})Z^{4}+15(11g_{2}-24\wp
^{2})\wp^{\prime}Z^{3} \\
&  -\tfrac{9}{4}(140\wp^{4}-245g_{2}\wp^{2}+190g_{3}\wp+21g_{2}^{2}%
)Z^{2}\label{z-n-4}\\
&  -(40\wp^{3}-163g_{2}\wp+125g_{3})\wp^{\prime}Z+\tfrac{3}{4}(25g_{2}%
-3\wp^{2})(\wp^{\prime})^{2}.
\end{align*}
}%
For general $n$, the expression of $Z_{r,s}^{(n)}(\tau)$ is too complicate to be written down.

Define
\[
F_{0}:=\{ \tau \in \mathbb{H}\ |\ 0\leqslant \  \text{Re}\  \tau \leqslant
1\  \text{and}\ |z-\tfrac{1}{2}|\geqslant \tfrac{1}{2}\},
\]
which is a fundamental domain of
\[
\Gamma_{0}(2):=\left \{  \left.
\begin{pmatrix}
a & b\\
c & d
\end{pmatrix}
\in SL(2,\mathbb{Z})\right \vert c\equiv0\text{ }\operatorname{mod}2\right \}.
\]
Here, as an application of Theorems \ref{no-sol-n} and \ref{LW-pre}, we have the following result.

\begin{theorem}
\label{neq-zero}Let $(r,s)\in \mathbb{R}^2\backslash \frac{1}{2}\mathbb{Z}^{2}$. Then $Z_{r,s}^{(n)}(\tau)\not =0$ for
any $\tau \in \partial F_{0}\cap \mathbb{H}$.
\end{theorem}

\begin{proof}
It does not seem that this assertion could be obtained directly from the expressions
of $Z_{r,s}^{(n)}(\tau)$ even for $n\leq 4$. Indeed, this lemma is a consequence of our PDE result.

Given $\tau \in \partial F_{0}\cap \mathbb{H}$. If $\tau \in i\mathbb{R}_{>0}$, then Theorems \ref{no-sol-n} and \ref{LW-pre} together
imply $Z_{r,s}^{(n)}(\tau)\not =0$ for any $(r,s)\in \mathbb{R}^{2}%
\backslash \frac{1}{2}\mathbb{Z}^{2}$.
If $\tau \in i\mathbb{R}_{>0}+1$, then
by applying $\gamma=%
\begin{pmatrix}
1 & -1\\
0 & 1
\end{pmatrix}
$ in property (ii), we have that $\tau-1\in i\mathbb{R}_{>0}$ and
\[
Z_{r,s}^{(n)}(\tau)=Z_{r+s,s}^{(n)}(\tau-1)\not =0\text{ for any }%
(r,s)\in \mathbb{R}^{2}\backslash \tfrac{1}{2}\mathbb{Z}^{2}.
\]
If $|\tau-\frac{1}{2}|=\frac{1}{2}$, then again by applying $\gamma=%
\begin{pmatrix}
1 & 0\\
-1 & 1
\end{pmatrix}
$ in property (ii)  we see that $\frac{\tau}{1-\tau}\in i\mathbb{R}_{>0}$ and
\[
(1-\tau)^{\frac{n(n+1)}{2}}Z_{r,s}^{(n)}(\tau)=Z_{r,r+s}^{(n)}(\tfrac{\tau}{1-\tau}%
)\not =0\text{ for any }(r,s)\in \mathbb{R}^{2}{\small \backslash}\tfrac{1}%
{2}\mathbb{Z}^{2}.
\]
This completes the proof.
\end{proof}

Theorem \ref{neq-zero} has important applications to studying the zero structure of $Z_{r,s}^{(n)}(\tau)$. By property (ii), we can restrict $\tau$ in the fundamental domain $F_{0}$ of $\Gamma
_{0}(2)$,
and by (i), we only need to consider $(r,s)\in \lbrack0,1]\times
\lbrack0,\frac{1}{2}]\backslash \frac{1}{2}\mathbb{Z}^{2}$.
Define four open triangles:
{\allowdisplaybreaks
\begin{align*}%
&\triangle_{0}:=\{(r,s)\mid0<r,s<\tfrac{1}{2},\text{ }r+s>\tfrac{1}{2}\},\\
&\triangle_{1}:=\{(r,s)\mid \tfrac{1}{2}<r<1,\text{ }0<s<\tfrac{1}{2},\text{
}r+s>1\},\\
&\triangle_{2}:=\{(r,s)\mid \tfrac{1}{2}<r<1,\text{ }0<s<\tfrac{1}{2},\text{
}r+s<1\},\\
&\triangle_{3}:=\{(r,s)\mid r>0,\text{ }s>0,\text{ }r+s<\tfrac{1}{2}\}.
\end{align*}
}%
In \cite{CKLW,CL-E2}, Theorem \ref{neq-zero} was applied to prove the following results.
\begin{theorem} \cite{CKLW} Let $(r,s)\in \lbrack0,1]\times \lbrack0,\frac{1}{2}]\backslash
\frac{1}{2}\mathbb{Z}^{2}$. Then $Z_{r,s}(\tau)=0$ has a solution $\tau$
in $F_{0}$ if and only if $(r,s)\in \triangle_{0}$. Furthermore, for any $(r,s)\in \triangle_{0}$, the zero $\tau \in F_{0}$ is unqiue and satisfies
$\tau \in \mathring{F}_{0}=F_{0}\setminus \partial F_0$.
\end{theorem}
\begin{theorem}\cite{CL-E2}
\label{thm20}Let $(r,s)\in \lbrack0,1]\times \lbrack0,\frac{1}{2}]\backslash
\frac{1}{2}\mathbb{Z}^{2}$. Then $Z_{r,s}^{(2)}(\tau)=0$ has a solution $\tau$
in $F_{0}$ if and only if $(r,s)\in \triangle_{1}\cup \triangle_{2}\cup
\triangle_{3}$. Furthermore, for any $(r,s)\in \triangle_{1}\cup \triangle
_{2}\cup \triangle_{3}$, the zero $\tau \in F_{0}$ is unqiue and satisfies
$\tau \in \mathring{F}_{0}$.
\end{theorem}

Among their applications back to the curvature equation (\ref{mfe1}), such results have other interesting applications. For example, Theorem \ref{thm20} can be used to completely determine the critical points of the Eisenstein series $E_2(\tau)$ of weight $2$; see \cite{CL-E2}. We will study the zero structure of $Z_{r,s}^{(n)}(\tau)$ for $n\in\{3,4\}$ via Theorem \ref{neq-zero} in future.

\appendix

\section{Spectral theory and finite-gap potential}

\label{Floquet-theory}

In this appendix, we recall the spectral theory for Hill's
equation with complex-valued potentials \cite{GW}, which will be applied in Lemma \ref{lemma5}.

Let $q(x)$ is a complex-valued continuous nonconstant periodic function of
period $\Omega$ on $\mathbb{R}$. Consider the following Hill's equation%
\begin{equation}
y^{\prime \prime}(x)+q(x)y(x)=Ey(x),\text{ \  \ }x\in \mathbb{R}. \label{eq2-1}%
\end{equation}
This equation has received an enormous amount of consideration due to its
ubiquity in applications as well as its structural richness; see e.g.
\cite{GW,GW2} and references therein for historical reviews.

Let $y_{1}(x)$ and $y_{2}(x)$ be any two linearly independent solutions of
equation (\ref{eq2-1}). Then so do $y_{1}(x+\Omega)$ and $y_{2}(x+\Omega)$ and
hence there exists a monodromy matrix $M(E)\in SL(2,\mathbb{C})$ such that
\[
(y_{1}(x+\Omega),y_{2}(x+\Omega))=(y_{1}(x), y_{2}(x))M(E).
\]
A solution of Hill's equation (\ref{eq2-1}) is called \textit{a Floquet
solution} if it is a eigenfunction of the monodromy matrix $M(E)$.
Define the \emph{Hill's discriminant} $\Delta (E)$ by
\begin{equation}
\label{trace}\Delta(E):=\text{tr}M(E),
\end{equation}
which is clearly an invariant of (\ref{eq2-1}), i.e. does not depend on the choice of linearly
independent solutions. This $\Delta(E)$ is an entire function and plays a
fundamental role since it encodes all the spectrum
information of the associated operator; see e.g. \cite{GW2} and references
therein. Indeed, we define
\begin{equation}
\mathcal{S}:=\Delta^{-1}([-2,2])=\{E\in \mathbb{C}\,|\, -2\leq \Delta(E)\leq2\}
\end{equation}
to be \textit{the conditional stability set} of the operator $L=\frac{d^{2}%
}{dx^{2}}+q(x)$. Since $q(x)$ is assumed to be continuous, $\mathcal{S}$ can
be characterized as
\[
\mathcal{S}=\{E\in \mathbb{C}\,|\, Ly=Ey\; \text{has a bounded solution on
$\mathbb{R}$}\}.
\]
Then it was proved in \cite{Rofe-Beketov} that \textit{this $\mathcal{S}$
coincides with the spectrum $\sigma(H)$ of the associated linear operator $H$
in $L^{2}(\mathbb{R}, \mathbb{C})$} (i.e. $H$ is defined as $Hf=Lf$, $f\in
H^{2,2}(\mathbb{R},\mathbb{C})$).

On the other hand, we define
\[
d(E):=\text{ord}_{E}(\Delta(\cdot)^{2}-4).
\]
Then it is well known (cf. \cite[Section 2.3]{Naimark}) that $d(E)$ equals \textit{the algebraic multiplicity of (anti)periodic
eigenvalues}. Let $c(E,x,x_{0})$ and $s(E,x,x_{0})$ be the special fundamental
system of solutions of (\ref{eq2-1}) satisfying by the initial values
\[
c(E,x_{0},x_{0})=s^{\prime}(E,x_{0},x_{0})=1,\ c^{\prime}(E,x_{0}%
,x_{0})=s(E,x_{0},x_{0})=0.
\]
Then we have
\[
\Delta(E)=c(E,x_{0}+\Omega,x_{0})+s^{\prime}(E,x_{0}+\Omega,x_{0}).
\]
Define
\[
p(E,x_{0}):=\text{ord}_{E}s(\cdot,x_{0}+\Omega,x_{0}),
\]%
\[
p_{i}(E):=\min \{p(E,x_{0}):x_{0}\in \mathbb{R}\}.
\]
It is known  that $p(E,x_{0})$ is the algebraic
multiplicity of a Dirichlet eigenvalue on the interval $[x_0, x_0+\Omega]$,
and $p_{i}(E)$ denotes the immovable part of $p(E,x_{0})$ (cf. \cite{GW}). It
was proved in \cite[Theorem 3.2]{GW} that $d(E)-2p_{i}(E)\geq0$. Define
\begin{equation}
D(E):=E^{p_{i}(0)}\prod \limits_{\lambda \in \mathbb{C}\backslash \{0\}} \left(
1-\tfrac{E}{\lambda}\right)  ^{p_{i}(\lambda)}. \label{eq3-2}%
\end{equation}
Let us recall the following important result proved in \cite{GW}.

\medskip \noindent \textbf{Theorem B.} \cite[Theorem 4.1]{GW} \textit{Assume
that $q(x)$ is a complex-valued continuous nonconstant periodic function of
period $\Omega$ on $\mathbb{R}$ and that equation (\ref{eq2-1}) has two
linearly independent Floquet solutions for all $E\in \mathbb{C}\setminus
\{E_{j}\}_{j=1}^{\tilde{m}}$ for some $\tilde{m}\in \mathbb{Z}_{\geq0}$ and
precisely one Floquet solution for each $E=E_{j}$. Then }

\textit{(i) $d(E)-2p_{i}(E)>0$ on a finite set $\{E_{j}\}_{j=1}^{m}$ including
$\{E_{j}\}_{j=1}^{\tilde{m}}$, $m\geq \tilde{m}$, and $d(E)-2p_{i}(E)=0$
elsewhere. The Wronskian of two nontrivial Floquet solutions which are
linearly independent on some punctured disk $0<|E-\lambda|<\varepsilon$ tends
to zero as $E\to \lambda$ if and only if $\lambda \in \{E_{j}\}_{j=1}^{m}$. }

\textit{(ii) $\sum_{j=1}^{m} (d(E_{j})-2p_{i}(E_{j}))=2g+1$ for some
$g\in \mathbb{Z}_{\geq0}$ and $q(x)$ is an algebro-geometric finite gap
potential associated with the compact (possibly singular) hyperelliptic curve
obtained upon one-point compactification of the curve
\begin{equation}
\label{hyper}F^{2}=R_{2g+1}(E):=\prod_{j=1}^{m} (E-E_{j})^{d(E_{j}%
)-2p_{i}(E_{j})}=C\frac{\Delta(E)^{2}-4}{D(E)^{2}}.
\end{equation}
Here $D(E)$ is seen in (\ref{eq3-2}) and $C$ is some nonzero constant. }

\textit{(iii) the spectrum $\sigma(H)=\mathcal{S}$ consists of finitely many
bounded spectral arcs $\sigma_{k}$, $1\leq k\leq \tilde{g}$ for some $\tilde
{g}\leq g$ and one semi-infinite arc $\sigma_{\infty}$ which tends to
$-\infty+\langle q\rangle$, with $\langle q\rangle=\frac{1}{\Omega}\int
_{x_{0}}^{x_{0}+\Omega}q(x)dx$, i.e.
\[
\sigma(H)=\mathcal{S}=\sigma_{\infty}\cup \cup_{k=1}^{\tilde{g}}\sigma_{k}.
\]
Furthermore, the finite end points of such arcs must be those $E\in
\{E_{j}\}_{j=1}^{m}$ with $d(E)$ odd. } \medskip

Remark that if we assume in addition that $q(x)$ is \textit{real-valued} in
Theorem B, then it is well-known (cf. \cite{GW,GW2}) that $R_{2g+1}(E)$ has
$2g+1$ distinct real zeros, denoted by $E_{2g}<E_{2g-1}<\cdots<E_{1}<E_{0}$,
and
\begin{equation}
\label{spectrum}\sigma(H)=\mathcal{S}=(-\infty, E_{2g}]\cup[E_{2g-1},
E_{2g-2}]\cup \cdots \cup[E_{1}, E_{0}],
\end{equation}
namely the spectrum has the so-called \emph{finite-gap property}, and so $q(x)$ is a so-called \emph{finite-gap potential}. In Section \ref{general-lame}, we show that even if $q(x)$ is \emph{not real-valued}, the
finite-gap property (\ref{spectrum}) might still hold in some special
situations; see Lemma \ref{lemma5}.

\bigskip

{\bf Acknowledgements} The research of the first author was supported by NSFC (No. 11701312).


\begin{thebibliography}{99}




\bibitem {CLMP}E. Caglioti, P. L. Lions, C. Marchioro and M. Pulvirenti;
\textit{A special class of stationary flows for two-dimensional Euler
equations: a statistical mechanics description.} Comm. Math. Phys.
\textbf{143} (1992), 501-525.

\bibitem {CLW}{C. L. Chai, C. S. Lin and C. L. Wang; \textit{Mean field
equations, Hyperelliptic curves, and Modular forms: I}. Camb. J. Math.
\textbf{3} (2015), 127-274.}

\bibitem {CLW4}C. C. Chen, C. S. Lin and G. Wang; \textit{Concentration
phenomena of two-vortex solutions in a Chern-Simons model}. Ann. Scuola Norm.
Sup. Pisa Cl. Sci. (5) \textbf{3} (2004), 367-397.

\bibitem {CKL}Z. Chen, T. J. Kuo and C. S. Lin; \textit{Nonexistence of
solutions for mean field equation on rectangular torus at critical parameter
$16\pi$}. Comm. Anal. Geom. to appear. arXiv:1610.01787v2 [math. AP] 2016.


\bibitem {CKL-2}Z. Chen, T. J. Kuo and C. S. Lin; \textit{Existence and
non-existence of solutions of the mean field equations on flat tori}. Proc.
Amer. Math. Soc. \textbf{145} (2017), 3989-3996.

\bibitem {CKLT}Z. Chen, T. J. Kuo, C. S. Lin and K. Takemura;
\textit{Real-root property of the spectral polynomial of the Treibich-Verdier
potential and related problems}. arXiv:1610.02216v1 [math. CA] 2016.

\bibitem {CKLW}Z. Chen, T. J. Kuo, C. S. Lin and C. L. Wang; \textit{Green
function, Painlev\'{e} VI equation, and Eisenstein series of weight one}. J.
Differ. Geom. to appear.

\bibitem {CL-E2} Z. Chen and C. S. Lin; \textit{Critical points of the classical Eisenstein series of weight two}. arXiv:1707.04804v1 [math. NT] 2017.

\bibitem {Chen-Lin2017} Z. Chen, C. S. Lin and X. Zhong; \textit{Exact number and non-degeneracy of critical points of multiple Green function
on rectangular tori}. in preparation.

\bibitem {Choe}K. Choe; \textit{Asymptotic behavior of condensate solutions in
the Chern-Simons-Higgs theory.} J. Math. Phys. \textbf{48} (2007).

\bibitem {EG}A. Eremenko and A. Gabrielov; \textit{Spherical Rectangles}.
Arnold Math. J. \textbf{2} (2016), 463-486.

\bibitem {EG1}A. Eremenko and A. Gabrielov; \textit{On metrics of curvature $1$ with four conic singularities on tori and on the sphere}.  Illinois J. Math. \textbf{59} (2015), 925-947.

\bibitem {GW}F. Gesztesy and R. Weikard; \textit{Picard potentials and Hill's
equation on a torus}. Acta Math. \textbf{176} (1996), 73-107.

\bibitem {GW1}F. Gesztesy and R. Weikard; \textit{Treibich-Verdier potentials
and the stationary (m)KdV hierarchy}. Math. Z. \textbf{219} (1995), 451-476.

\bibitem {GW2}F. Gesztesy and R. Weikard; \textit{Floquet theory revisited}.
Differential equations and mathematical physics, 67-84, Int. Press, Boston,
MA, 1995.

\bibitem {Lin-CDM}C.S. Lin; \textit{Green function, mean field equation and
Painlev\'{e} VI equation}. Current Developments in Mathematics. 2015, 137-188.

\bibitem {LW}{C. S. Lin and C. L. Wang; \textit{Elliptic functions, Green
functions and the mean field equations on tori}. Ann. Math. \textbf{172}
(2010), 911-954.}

\bibitem {CLW2}C. S. Lin and C. L. Wang; \textit{Mean field equations,
Hyperelliptic curves, and Modular forms: II}. J. \'{E}c. polytech. Math.
\textbf{4} (2017), 557-593.

\bibitem {LW4}C. S. Lin and C. L. Wang; \textit{On the minimality of extra
critical points of Green functions on flat tori}, Int. Math. Res. Not. to
appear, 2016.

\bibitem {LY}C. S. Lin and S. Yan; \textit{Existence of bubbling solutions for
Chern-Simons model on a torus.} Arch. Ration. Mech. Anal. \textbf{207} (2013), 353-392.

\bibitem{Naimark} M. A. Naimark; \textit{Linear differential operators} (Frederick Ungar, New York, 1967).

\bibitem {NT1}M. Nolasco and G. Tarantello; \textit{Double vortex condensates
in the Chern-Simons-Higgs theory.} Calc. Var. PDE. \textbf{9} (1999), 31-94.

\bibitem {Rofe-Beketov}F. Rofe-Beketov; \textit{The spectrum of
non-selfadjoint differential operators with periodic coeffients.} Soviet Math.
Dokl. \textbf{4} (1963) 1563-1566.

\bibitem {Tak1}K. Takemura; \textit{The Heun equation and the
Calogero-Moser-Sutherland system I: the Bethe Ansatz method}. Comm. Math.
Phys. \textbf{235} (2003), 467-494.

\bibitem {Tak2}K. Takemura; \textit{The Heun equation and the
Calogero-Moser-Sutherland system II: the perturbation and the algebraic
solution}. Electron. J. Differential Equations. \textbf{2004} (2004), 1--30.

\bibitem {Tak3}K. Takemura; \textit{The Heun equation and the
Calogero-Moser-Sutherland system III: the finite gap property and the
monodromy}. J. Nonlinear Math. Phys. \textbf{11} (2004), 21--46.

\bibitem {Tak4}K. Takemura; \textit{The Heun equation and the
Calogero-Moser-Sutherland system IV: the Hermite-Krichever Ansatz}. Comm.
Math. Phys. \textbf{258} (2005), 367--403.

\bibitem {Tak5}K. Takemura; \textit{The Heun equation and the
Calogero-Moser-Sutherland system V: generalized Darboux transformations}. J.
Nonlinear Math. Phys. \textbf{13} (2006), 584--611.

\bibitem {Tak-MathZ}K. Takemura; \textit{The Hermite-Krichever Ansatz for
Fuchsian equations with applications to the sixth Painlev\'{e} equation and to
finite gap potentials}. Math. Z. \textbf{263} (2009), 149-194.

\bibitem {TV}A. Treibich and J. L. Verdier; \textit{Revetements exceptionnels
et sommes de 4 nombres triangulaires}. Duke Math. J. \textbf{68} (1992), 217-236.


\bibitem {W1998}R. Weikard; \textit{On Hill's equation with a singular
complex-valued potential}. Proc. Lond. Math. Soc. \textbf{76} (1998), 603-633.

\end{thebibliography}
\end{document}